\documentclass[10pt]{article}

\usepackage{url}
\usepackage{mathtools}
\usepackage{amssymb}
\usepackage{amsthm}
\usepackage{empheq}
\usepackage{latexsym}
\usepackage{enumitem}
\usepackage{eurosym}
\usepackage{dsfont}
\usepackage{appendix}
\usepackage{color} 
\usepackage[unicode]{hyperref}
\usepackage{frcursive}
\usepackage[utf8]{inputenc}
\usepackage[T1]{fontenc}
\usepackage{geometry}
\usepackage{multirow}
\usepackage[colorinlistoftodos]{todonotes}
\usepackage{lmodern}
\usepackage{anyfontsize}
\usepackage{stmaryrd}
\usepackage{natbib}
\usepackage{cleveref}

\usepackage{amsbsy}

\setcitestyle{numbers,open={[},close={]}}

\definecolor{red}{rgb}{0.7,0.15,0.15}
\definecolor{green}{rgb}{0,0.5,0}
\definecolor{blue}{rgb}{0,0,0.7}
\hypersetup{colorlinks, linkcolor={red},citecolor={green}, urlcolor={blue}}
			
\makeatletter \@addtoreset{equation}{section}

\newtheorem{theorem}{Theorem}[section]
\newtheorem{assumption}[theorem]{Assumption}

\newtheorem{lemma}[theorem]{Lemma}
\newtheorem{proposition}[theorem]{Proposition}

\newtheorem{definition}[theorem]{Definition}
\newtheorem{remark}[theorem]{Remark}

\def\no{\noindent}
\def\beq{\begin{eqnarray}}
\def\eeq{\end{eqnarray}}
\def\be*{\begin{eqnarray*}}
\def\ee*{\end{eqnarray*}}

\def \E{\mathbb{E}}
\def \F{\mathbb{F}}

\def \L{\mathbb{L}}

\def \N{\mathbb{N}}

\def \P{\mathbb{P}}
\def \Q{\mathbb{Q}}
\def \R{\mathbb{R}}

\def\Ac{{\cal A}}
\def\Bc{{\cal B}}
\def\Cc{{\cal C}}

\def\Ec{{\cal E}}
\def\Fc{{\cal F}}

\def\Lc{{\cal L}}

\def\Pc{{\cal P}}

\def\Sc{{\cal S}}

\def\Wc{{\cal W}}

\def\x{\times}
\def\eps{\varepsilon}

\def\Om{\Omega}

\def\Omh{\widehat{\Omega}}

\def\0{\mathbf{0}}
\def \Ec{\mathcal{E}}

\def \xb{\mathbf{x}}

\def \mub{\overline{\mu}}

\def\normeL2#1{\left\|{#1}\right\|_{L^2}}

\def\Xh{\widehat X}

\def \Lim{\displaystyle\lim}

\def \hax {\hat{X}}

\def \1{\mathds{1}}
\def \d{{\rm d}}
\def \ep{\hbox{ }\hfill$\Box$}

\def \tq {g_0^{c'}}

\DeclareUnicodeCharacter{014D}{\=o}
 \title{Mean field game of mutual holding with defaultable agents, 
 \\ and systemic risk}

\author{
 Mao Fabrice Djete\footnote{\'Ecole Polytechnique Paris, Centre de Math\'ematiques Appliqu\'ees, mao-fabrice.djete@polytechnique.edu. This work benefits from the financial support of the Chairs {\it Financial Risk} and {\it Finance and Sustainable Development}} \and Gaoyue Guo\thanks{Universit\'e Paris--Saclay  CentraleSup\'elec, MICS and CNRS FR-3487, gaoyue.guo@centralesupelec.fr. }
    \and Nizar Touzi\thanks{Ecole Polytechnique Paris, Centre de Math\'ematiques Appliqu\'ees,
        nizar.touzi@polytechnique.edu. This work benefits from the financial support of the Chairs {\it Financial Risk} and {\it Finance and Sustainable Development}. }
    }
             \date{\today}

\begin{document}

\maketitle
 
\begin{abstract}
We introduce the possibility of default in the mean field game of mutual holding of \citeauthor*{touzidjete21} \cite{touzidjete21}. This is modeled by introducing absorption at the origin of the equity process. We provide an explicit solution of this mean field game. Moreover, we provide a particle system approximation, and we derive an autonomous equation for the time evolution of the default probability, or equivalently the law of the hitting time of the origin by the equity process. The systemic risk is thus described by the evolution of the default probability.
\end{abstract}

\vspace{3mm}
\no{\bf Keywords.} Mean field McKean-Vlasov stochastic differential equation, hitting time coupling, mean field game, default probability. 

\vspace{3mm}
\no{\bf MSC2010.} 60K35, 60H30, 91A13, 91A23, 91B30.

\section{Introduction}\label{sec:intro}

This paper addresses the strategic interaction between economic agents in order to build a structural model for the equilibrium equity processes observed on financial markets, and to better understand the most important determinants of systemic risk and default contagion.

Among the huge economic literature on this problem, we refer to \citeauthor*{allen2000financial} \cite{allen2000financial}, \citeauthor*{giesecke2004cyclical} \cite{giesecke2004cyclical}, \citeauthor*{shin2009securitisation} \cite{shin2009securitisation}, \citeauthor*{eisenberg2001systemic} \cite{eisenberg2001systemic}, and \citeauthor*{acemoglu2015systemic} \cite{acemoglu2015systemic}. We also refer to the financial regulation works on system-wide stress tests, which emerged from the last financial crisis, see \citeauthor*{aikman2019system} \cite{aikman2019system}. The main purpose of this paper is to provide an economic foundation for the agent based dynamics used in the simulation of the shock propagation addressed in this literature very connected with highly relevant practical regulatory decisions.

This paper builds on the mutual holding model introduced by \citeauthor*{touzidjete21} \cite{touzidjete21} which focuses on the strategic connection between economic agents through equity cross holding, while ignoring other (important) aspects as debt interconnection, which are still left for future work. Our main task in this paper is to introduce the possibility of default, which represents a major technical difficulty left open in the first work of \cite{touzidjete21} in this direction. 

We assume that default occurs at the first time that the underlying equity value hits the origin.  Given the idiosyncratic risk process driven by independent Brownian motions $W^i$, and defined by the drift and diffusion coefficients $b$ and $\sigma$, the microscopic dynamics of the equity value processes $(X^1,\ldots, X^N)$ of $N$ homogeneous defaultable economic agents are defined by,
\be*
\begin{split}
dX^i_t 
& = 
\1_{\{\tau_i>t\}}\Big\{b_t(X^i_t, \widehat\mu^N_{t})\mathrm{d}t 
                               + \sigma_t(X^i_t, \widehat\mu^N_t)\mathrm{d}W^i_t
 + \frac{1}{N}\sum_{j=1}^N \1_{\{\tau_j>t\}} 
           \big[\pi_t(X^i_t, X^j_t)\mathrm{d}X^j_t 
                  -\pi_t(X^j_t, X^i_t)\mathrm{d}X^i_t
           \big]
 \Big\},
\end{split} 
\ee*
where, similar to \cite{touzidjete21}, we denote $\widehat\mu^{N}_t := \frac{1}{N}\sum_{j=1}^N\delta_{X^j_t}$, and
\begin{itemize}
\item $\pi(t,X^k_t, X^\ell_t)$ denotes the number of shares of agent $\ell$ held by agent $k$ at time $t$, which is assumed to be a decision variable at the hand of Agent $k$,
\item the coefficients $b$ and $\sigma$ of the idiosyncratic risk process illustrate the anonymity with the population of agents by its dependence on the the empirical measure $\widehat\mu^{N}_t$,
\item as a new feature of the present paper, $\tau_i:=\inf\{t\ge 0: X^i_t\le 0\}$ is the default time of the $i-$th economic agent, which is modeled here by an absorption at the origin, meaning that no possibly recovery is considered in this model.
\end{itemize}
The natural mean field limit $N\to\infty$ suggested by this microscopic model is: 
\begin{equation} \label{controlled0}
dX_t 
\!=\!
\1_{\{\tau>t\}}\!
\Big\{b_t(X_t, \mu_t)\mathrm{d}t 
        \!+\!\sigma_t(X_t, \mu_t)\mathrm{d}W_t 
        \!+\!\widehat\E^{\mu}\big[\1_{\{\hat{\tau}>t\}} 
                                           \pi_t(X_t, \widehat{X}_t)
                                           \mathrm{d}\widehat{X}_t\big] 
        \!-\! \widehat\E^{\mu}\big[ \1_{\{\hat{\tau}>t\}} 
                                             \pi_t(\widehat{X}_t, X_t)\big]
                                             \mathrm{d}X_t
\Big\},
\end{equation}
where $(\widehat{X},\hat\tau)$ is an independent copy of $(X,\tau)$, and $\widehat\E^{\mu}$ denotes expectation operator conditional to $(X,\tau)$. 

Our first objective in this paper is to provide a solution of the mean field game of mutual holding as introduced in \cite{touzidjete21} with the above additional feature of defaultable agents. The coupling of the population through the hitting time of zero introduces additional technical difficulties as the equilibrium dynamics turn out to be characterized by nontrivial singularities of the coefficients (in Wasserstein space) which were not encountered in the previous non-defaultable model of \cite{touzidjete21}. Theorem \ref{thm:main} provides an explicit characterization of a mean field game equilibrium of mutual holding, thus extending the non-defaultable setting of \cite{touzidjete21}.

Our second main result concerns the systemic risk of the connected agents through mutual holding at equilibrium as described by the time evolution of the default probability or, equivalently, of the law of the default time of a representative agent. In addition to the expected particle system approximation induced by the propagation of chaos results of \Cref{thm:approx_regularity} and \Cref{prop:approx_poc}, we provide in \Cref{thm:more} an original expression of the equilibrium marginal law of the equity process in terms of the marginal law of the corresponding non-mean field and non-absorbed process. In particular, when the coefficients of the idiosyncratic risk are not coupled, this equation induces an autonomous equation for the default probability, which doest not require the full knowledge of the law of the equilibrium equity value. Due to the singularity of the coefficients of the equilibrium dynamics, we are unfortunately not able to exploit further this equation. However, when the drift coefficient of the idiosyncratic risk $b,\sigma$ have constant sign and are free from interaction, \Cref{thm:uniqueness} states that the default probability is uniquely defined by this equation, and we then deduce that our equilibrium dynamics have a unique strong solution. We obtain this result as a consequence of the uniqueness of a solution to our autonomous equation for the evolution of the default probability, which is proved by some delicate estimates on diffusions densities by the so-called Parametrix method.  

We notice that similar McKean-Vlasov stochastic differential equations (SDE) coupled through hitting time of the origin was considered in recent models by \citeauthor*{hambly2019spde} \cite{hambly2019spde}, \citeauthor*{nadtochiy2019particle},  \cite{nadtochiy2019particle}, \cite{nadtochiy2020particle}, \citeauthor*{DNS2019} \cite{DNS2019}, \citeauthor*{CRS2020} \cite{CRS2020}, \citeauthor*{bayraktar2020mckean}  \cite{bayraktar2020mckean}, \cite{bayraktar2022}. We also refer to the related literature modeling interacting economic agents by McKean-Vlasov SDEs, see e.g. \citeauthor*{garnier2013large} \cite{garnier2013large} \citeauthor*{carmona2013mean} \cite{carmona2013mean},  \citeauthor*{sun2022}  \cite{sun2022}.

The paper is organized as follows. The rigorous formulation of the mean field game of mutual holding with defaultable agents is reported in Section \ref{sec:pbfm}. Our main results are collected in Section \ref{sec:main}. The existence result of our equilibrium McKean-Vlasov SDE is proved in Section \ref{sec:proof1} by means of appropriate approximation arguments. In particular, Subsection \ref{sec:autonomous} contains the proof of our autonomous equation for the default probability. Finally, Section \ref{sec:uniqueness} specializes to the special case of constant sign non-interacting coefficients $b,\sigma$ of the idiosyncratic risk, and prove that the equilibrium McKean-Vlasov SDE has a unique strong solution in this case.

\medskip

\noindent {\bf Notations}\quad 
For a generic Polish space $E$, we denote $\Bc_E$ the collection of all Borel subsets of $E$, and by $\Pc(E)$ the collection of all probability measures $m:A\in\Bc_E\longmapsto m[A]\in[0,1]$. We denote similarly $m(f) := \int_{E}f\d m$, $m-$integrable function $f:E\longrightarrow\R$. 

For $q>0$, the subset $\Pc_{q}(E)\subset\Pc(E)$ consists of all measures $m$ with finite $q$-th moment $\int_E|x|^qm(\d x)<\infty$, and we denote by $\Wc_q$ the corresponding $q-$Wasserstein distance.

Finally, $C(E)$ is the space of continuous real-valued functions on $E$, and $C_b(E)$ denotes the corresponding subset of bounded functions. 

\section{Problem formulation}\label{sec:pbfm}

Let $\widehat{\Om}:=C(\R_+)$ be  endowed with the compact convergence topology. Denote by $\widehat{X}=(\widehat{X}_t)_{t\ge 0}$ its coordinate process and by $\widehat{\F}^0=(\widehat{\Fc}^0_t)_{t\ge 0}$ its natural filtration, i.e. $\widehat{X}_t(\hat{\mathsf x}) :=\hat{\mathsf x}(t)$ for  $\hat{\mathsf x}\in\widehat{\Om}$ and $\widehat{\Fc}^0_t:=\sigma(\widehat{X}_s, s\le t)$. Define further  $\widehat{\F}=(\widehat{\Fc}_t)_{t\ge 0}$ to be the  filtration given by   $\widehat{\Fc}_t:=\lim_{s\downarrow t}\widehat{\Fc}_s^U$ with $\widehat{\Fc}_t^U:=\cap_{\mu\in\Pc(\widehat{\Om})} (\widehat{\Fc}_t^0)^{\mu}$, where $(\widehat{\Fc}_t^0)^{\mu}$ denotes the completion of $\widehat{\Fc}_t^0$ by $\mu$.

\subsection{Equilibrium dynamics}\label{ssec:equidyn}

Throughout this paper, we fixed a probability distribution $\rho\in \Pc_{q}(\R_+)$ such that
$\mbox{Supp}(\rho)\subset (0,\infty)=:\R_+^*$, for some $q>1$, and we denote by $\Pc_S\subset \Pc(\widehat{\Om})$ the subset of probability measures $\mu$ under which $\widehat{X}$ is an It\^o process, absorbed at zero, and satisfying the dynamics 
\be* 
\mu\circ \widehat{X}_0^{-1}=\rho
&\mbox{and}&
\Xh_t = \Xh_0 + \int_0^t \1_{\{\widehat{X}_s>0\}}  B^{\mu}_s \mathrm{d}s + \int_0^t \1_{\{\widehat{X}_s>0\}} \Sigma^{\mu}_s \mathrm{d}W^{\mu}_s,
 \quad  \forall t\ge 0,
\ee* 
for some $\mu-$scalar Brownian motion $W^{\mu}$, and $\widehat{\F}-$progressively maps $B^{\mu}, \Sigma^{\mu}: \R_+\times \Omh\to \R$ satisfying 
\be*
\widehat\E^{\mu} \left[\int_0^t |B_s^{\mu}|^2 + |\Sigma_s^{\mu}|^2\mathrm{d}s \right] < \infty,\quad \forall 
t\ge 0.
\ee*
Introducing its first hitting time at the origin
$\hat\tau:=\inf\{t\ge 0: \Xh_t\le 0\}$,   
we may rewrite the last stochastic differential equation (SDE) as
\be* 
\Xh_t 
=
\Xh_0 + \int_0^t \1_{\{\hat\tau>s\}}  B^{\mu}_s \mathrm{d}s + \int_0^t \1_{\{\hat\tau>s\}} \Sigma^{\mu}_s \mathrm{d}W^{\mu}_s 
=\Xh_0 + \int_0^{\hat\tau \wedge t}  B^{\mu}_s \mathrm{d}s + \int_0^{\hat\tau \wedge t}   \Sigma^{\mu}_s \mathrm{d}W^{\mu}_s, 
 \quad \forall  t\ge 0.
\ee*
In order to introduce our mutual holding problem, we introduce the set of holding strategies $\Ac$ consisting of all measurable functions
$
\pi : [0, T] \times\R \times\R \longrightarrow [0, 1]$. Here, $\pi(t,x,y)$ represents the proportion of Agent $y$'s equity held by Agent $x$ at time $t$.

For $\mu\in\Pc_S$, and $\pi\in\Ac$, the mutual holding problem discussed in the introduction leads to the mean field SDE driven by Brownian motion $W$ on some filtered complete probability space $(\Omega,\Fc,\P)$: 
\begin{align*}
\begin{split}
X_t & = X_0 + \int_0^{\tau \wedge t} b(s,X_s, \mu_{s})\mathrm{d}s + \int_0^{\tau\wedge t} \sigma(s,X_s, \mu_{s})\mathrm{d}W_s \\
 &~~~~~~ + \widehat\E^{\mu}\left[\int_0^{\tau \wedge t}\1_{\{\hat{\tau}>s\}} \pi(s, X_s, \widehat{X}_s)\mathrm{d}\widehat{X}_s\right] -  \int_0^{\tau\wedge t}\widehat\E^{\mu}\left[ \1_{\{\hat{\tau}>s\}} \pi(s, \widehat{X}_s, X_s)\right]\mathrm{d}X_s,\quad \forall t\ge 0,
\end{split} 
\end{align*}
where $\tau:=\inf\{t\ge 0: X_t\le 0\}$ and we use the notation 
\be* 
\widehat\E^\mu[F(X,\widehat{X})] :=\int_{\widehat\Om}F(X,\mathbf{\widehat x})\mu(\d\mathbf{\widehat x})
\ee* 
for all measurable functions $F: \widehat\Om^2\to\R$ with appropriate integrability. Rewriting the last SDE in differential form and collecting the terms in ``$\d X_t$'', we obtain for all $t\ge 0$
\beq \label{MF-SDE}
\mathrm{d}X_t  
= 
\1_{\{X_t>0\}} \bigg(\frac {b(t,X_t, \mu_{t}) 
                                       +\widehat\E^{\mu}\big[ \1_{\{\widehat{X}_t>0\}} 
                                                                           \pi(t, X_t, \widehat{X}_t)B^{\mu}_t\big]
                                      }
                                     {1 + \widehat\E^{\mu}\big[ \1_{\{\widehat{X}_t>0\}} \pi(t, \widehat{X}_t, X_t)\big]}
\mathrm{d}t 
+ 
\frac{\sigma(t,X_t, \mu_{t})}
       {1 + \widehat\E^{\mu}\big[ \1_{\{\widehat{X}_t>0\}} \pi(t, \widehat{X}_t, X_t)\big]} 
\mathrm{d}W_t \bigg).
\eeq
Define $\Pc_S(\pi)\subset \Pc_S$ to be the subset of $\mu$ such that \eqref{MF-SDE} has a weak solution $X$ with $\Lc(X)=\mu$. Hence, by identification, $B^{\mu}_t=B(t, \hat{\mathsf x}(t),\mu_t)$ and $\Sigma^\mu_t=\Sigma(t, \hat{\mathsf x}(t),\mu_t)$ must hold with
$$
    B(t,x,\mu_t) = \frac {b(t,x, \mu_{t}) + \widehat\E^{\mu}\left[ \1_{\{\widehat{X}_t>0\}} \pi(t, x, \widehat{X}_t)B(t, \widehat X_t,\mu_t)\right]}{1 + \hat\E^{\mu}\left[ \1_{\{\widehat{X}_t>0\}} \pi(t, \widehat{X}_t, x)\right]}, 
    ~~ \Sigma(t,x,\mu_t) = \frac{\sigma(t,x, \mu_{t})}{1 +\widehat\E^{\mu}\left[ \1_{\{\widehat{X}_t>0\}} \pi(t, \widehat{X}_t, x)\right]}.
$$

\subsection{Mean field game formulation}
 \label{ssec:MFGformulation}

Next, we introduce a mean field game (MFG) formulation to determine the mutual holding strategy. For $\pi\in\Ac$ and $\mu\in\Pc_\Sc(\pi)$, we introduce the deviation of the representative agent from $\pi\in\Ac$ to an alternative strategy $\beta\in\Ac$ by means of an equivalent probability measure $\P^\beta_{\pi,\mu}$ defined via  
\beq 
 \frac{\mathrm{d}\P^\beta_{\pi,\mu}}{\mathrm{d}\P}\Big |_{\Fc_t}
 =
 \exp\left({\int_0^t \psi_s\mathrm{d}W_s-\frac12\int_0^t  \psi_s^2\mathrm{d}s}\right)=:\Ec_t(\psi),\quad \forall t\ge 0,
\label{psi}
\eeq
where $\Fc_t:=\sigma(X_s, 0\le s\le t)$ and the stochastic process $\psi\equiv \psi^\beta_{\pi,\mu}$ is given as  
 \beq \label{psi}
 \psi_t
 :=
 \frac{\widehat\E^\mu\big[\1_{\{\hax_t>0\}}(\beta-\pi)(t,X_t,\hax_t)B(t,\hax_t, 
\mu_t)\big]}{\sigma(t,X_t,\mu_t)}.
 \eeq 
To ensure the wellposedness of $\psi$ and $\P^\beta_{\pi,\mu}$,  
we assume that the diffusion coefficient $\sigma$ is bounded away from zero. It follows from Girsanov's theorem that $(W^{{\pi,\mu,\beta}}_t:=W_t-\int_0^t \psi_s\mathrm{d}s)_{t\ge 0}$ is a Brownian motion under $\P^\beta_{\pi,\mu}$, so that the $\P^\beta_{\pi,\mu}-$dynamics of the  process $X$ are rewritten as
\begin{align*}
\begin{split}
X_t & = X_0+\int_0^{\tau \wedge t} b(s,X_s,\mu_s)\mathrm{d}s
 +\int_0^{\tau \wedge t} \sigma(s,X_s,\mu_s)\mathrm{d}W^{\pi,\mu,\beta}_s
  \\
 &~~~~~~ +\widehat\E^\mu\bigg[ \int_0^{\tau \wedge t} \1_{\{\widehat X_s>0 \}}\beta(s,X_s,\widehat X_s) \mathrm{d}\widehat X_s \bigg]
 - \int_0^{\tau \wedge t} \widehat\E^\mu\left[\1_{\{\widehat X_s>0 \}}\pi(s,\widehat X_s,X_s)\right] \mathrm{d}X_s,
\end{split} 
\end{align*}
thus mimicking the controlled dynamics in \eqref{controlled0}. Therefore, the representative agent seeks for the optimal mutual holding strategy by maximizing the criterion
 $J_{\pi,\mu}(\beta):=\E^{\P^\beta_{\pi,\mu}}[U(X_T)]$ 
 over $\Ac$, where $T>0$ is an arbitrary time horizon and $U:\R_+\to\R$ denotes a non-decreasing utility function.
\begin{definition}\label{def:MFG}
A pair $(\pi,\mu)\in\Ac\times\Pc_\Sc$ is called an MFG  equilibrium of the mutual holding problem if $\mu\in\Pc_\Sc(\pi)$ and $J_{\pi,\mu}(\pi)=\sup_{\beta\in\Ac} J_{\pi,\mu}(\beta)$.
\end{definition}

\section{Main results}
\label{sec:main}

The paper is concerned with the existence of an MFG equilibrium for the mutual holding problem where, in contrast with the previous work \cite{touzidjete21}, we introduce here the default risk upon hitting the zero equity value. Our main results show that the dynamics of the equilibrium equity process is defined by the following mean field SDE
\begin{align} \label{eq:MFG-equilibriumSDE}
    \mathrm{d}X_t
    =
    \1_{\{X_t>0 \}} \Big[ B(t,X_t,\mu_t) \mathrm{d}t + \Sigma(t,X_t,\mu_t) \mathrm{d}W_t \Big],\quad \mbox{with }\Lc(X_0)=\rho \mbox{ and } \mu_t=\Lc(X_t),\quad \forall t\ge 0, 
\end{align}
where $B,\Sigma:\R_+ \x \R \x \Pc_1(\R_+) \to \R$ are defined by
\begin{eqnarray}\label{BSigma}
\hspace{-4mm}B(t,x,m)
:=
\frac{(b+c_1)^+(t,x,m)}{1+m[\R_+^*]}-(b+c_1)^-(t,x,m) 
&\hspace{-2mm}\mbox{and}&\hspace{-2mm}
\Sigma(t,x,m)
:=
\frac{\sigma(t,x,m)}{1+m[\R_+^*] \1_{\{ B(t,x,m) > 0 \}}},
\end{eqnarray}
and $c_1: \R_+\x \Pc_1(\R_+) \to \R_+ $ is, in view of Lemma \ref{lemma:existence_c}, the unique solution of the equation
\beq\label{eq:charac_c}
    c_1(t,m)=\frac{1}{1+m[\R_+^*]}\int_{(0,\infty)} \big(c_1(t,m)+b(t,y,m)\big)^+m(\mathrm{d}y).
\eeq 
The measurable functions $b,\sigma:\R_+ \x \R \x \Pc_1(\R_+) \longrightarrow \R$ stand for the drift and volatility of the idiosyncratic equity process. Notice that both coefficients $B$ and $\Sigma$ exhibit singularities in the variables $x$ and $m$, so that the wellposedness of the last mean field SDE will be the first issue to be clarified.

Notice that the expressions \eqref{BSigma} and \eqref{eq:charac_c} are closely related to the corresponding expressions in \cite{touzidjete21} where the equity process was not absorbed at zero, i.e. no default, so that $m[\R_+^*]$ is replaced in the above expression by $m[\R]=1$ in their context, and our $c_1$ in \eqref{eq:charac_c} coincides with their map $c$.

We also observe that the coefficients of the last mean field SDE exhibit a dependence on $\mu_t[\R_+^*]=\mu[\tau>t]=1-F_\tau(t)$, where $F_\tau$ is the cdf of the hitting time of the origin $\tau$. This is a similar feature to the models studied by \citeauthor*{hambly2019spde} \cite{hambly2019spde}, \citeauthor*{nadtochiy2019particle} \cite{nadtochiy2019particle}, \cite{nadtochiy2020particle}, \citeauthor*{nadtochiy2019particle} \cite{DNS2019}, \citeauthor*{CRS2020} \cite{CRS2020}, \citeauthor*{bayraktar2020mckean} \cite{bayraktar2020mckean}, \cite{bayraktar2022}. 

\begin{remark}\label{rem:bconstantsign}
By a straightforward computation, we see that the expressions of the equilibrium coefficients take the following simple form when the idiosyncratic drift coefficient has a constant sign: 
\begin{itemize} 
\item If $b\le 0$, then $B=b$ and $\Sigma=\sigma$;
\item If $b> 0$, then $B(t,x,m)=\displaystyle\frac{b(t,x,m)+\int_{(0,\infty)}b(t,y,m)m(\d y)}{1 +m[\R_+^*]}$ and $\Sigma(t,x,m)=\displaystyle\frac{\sigma(t,x,m)}{1 +m[\R_+^*]}$;
\item If $b\equiv b(t,m)$ with constant sign, then $B=b$ and $\Sigma(t,x,m)=\displaystyle\frac{\sigma(t,x,m)}{1+m[\R_+^*] \1_{\{b(t,m)>0\}}}$.
\end{itemize}
\end{remark}

\subsection{Existence of an MFG equilibrium}\label{ssec:exisence}

We start by stating the conditions on the maps $b$ and $\sigma$.
\begin{assumption}\label{assum:bsigma}
For every $T>0$, $b,\sigma$ are Lipschitz in $(x,m)$, uniformly in $t\in[0,T]$, and $\sigma$ is bounded from below away from zero, uniformly in $(t,x,m)$. In addition,
\begin{itemize}
\item[\rm (i)] either for each $m\in \Pc_1(\R_+)$, $b(t,x,m) \le 0,$ for a.e. $(t,x) \in [0,T] \x\R$,
\item[\rm (ii)]or for all $\eta\ge 0$, $m\in\Pc_1(\R_+),$ and for a.e. $t \in [0,T],$ the Borel set
\be*
    \ell (t,m,\eta)
    :=
    \big\{
        x \in \R:(x',m') \mapsto\1_{\{b(t,x',m') + \eta > 0 \}}\mbox{ is continuous at the point } (x,m)
    \big\}
\ee* 
has full Lebesgue measure in $\R$, i.e. its complement is Lebesgue--negligible.
\end{itemize}
\end{assumption}

Our first result ensures the existence of solutions to \eqref{eq:MFG-equilibriumSDE}, and states the existence of an MFG equilibrium with a complete characterization of the corresponding equilibrium mutual holding strategy.  

\begin{theorem}\label{thm:main}
Let Assumption \ref{assum:bsigma} hold. Then for every $T>0, q>1$, and all initial law  $\rho\in\Pc_q(\R_+)$:

\vspace{1mm}

{\rm (i)} the mean field SDE \eqref{eq:MFG-equilibriumSDE} has a weak solution $X$ on $[0,T]$, with $\E[\sup_{0\le t\le T}|X_t|^q]<\infty$;

\vspace{1mm}

{\rm (ii)} in addition, if $U$ is nondecreasing and Lipschitz, any weak solution $X$ of \eqref{eq:MFG-equilibriumSDE} induces an MFG equilibrium $(\pi^\ast,\Lc(X))$ of the mutual holding problem where  $\pi^\ast(t,x,y):=\1_{\{B(t,y,\Lc(X_t))> 0\}}$.
\end{theorem}
The proof of Theorem \ref{thm:main} (i) is based on an approximation argument reported in Section \ref{sec:proof1}, and differs from that in \cite{touzidjete21} by the additional difficulty induced by the lack of regularity (in Wasserstein distance) of $B$ and $\Sigma$ in $m$ and the possible degeneracy of the SDE \eqref{eq:MFG-equilibriumSDE}. As for (ii), we follow the same argument as in \cite{touzidjete21}, which we now report for completeness.

\proof[Proof of Theorem \ref{thm:main} {\rm (ii)}]
Fix an arbitrary (weak) solution $X$ to \eqref{eq:MFG-equilibriumSDE} and define the pair $(\pi^\ast,\mu)$, where $\pi^\ast(t,x,y):=\1_{\{ B(t,y,\mu_t) > 0\}}$, $\mu:=\Lc(X)$ and $\mu_t:=\Lc(X_t)$. Consider the optimization problem
 $V_0:=\sup_{\beta\in\Ac} J_{\pi^*\!,\mu}(\beta)$, where we recall that $J_{\pi^*\!,\mu}(\beta)
 :=
 \E^{\P_{\pi^*\!,\mu}^\beta}\big[U(X_T)\big]$  
and the dynamics of $X$ in terms of the $\P_{\pi^*\!,\mu}^\beta$--Brownian motion $W^{\pi^*\!,\mu,\beta}$ are given by
 \be*
 \mathrm{d}X_t
 =
 \1_{\{ X_t > 0 \}}\frac{\big[b(t,X_t,\mu_t)+\int_{(0,\infty)}\beta(t,X_t,y)B(t,y,\mu_t)\mu_t(\mathrm{d}y)\big]\mathrm{d}t+\sigma(t,X_t,\mu_t)\mathrm{d}W^{\pi^*\!,\mu,\beta}_t}
        {1+\mu_t(\R_+^*)\1_{\{B(t,X_t,\mu_t)> 0\}}},
 ~~ 
 t\in [0,T]. 
 \ee*
The Hamiltonian corresponding to this control problem is:
 \be*
 H_t(x,m,z)
 :=
  \1_{\{ x > 0\}}\frac{ z\,b(t,x,m)+\int_{(0,\infty)}(zB(t,y,m))^+m(\mathrm{d}y)}
        {1+m(\R_+^*)\1_{\{B(t,x,m)> 0\}}},
 ~~
 (t,x,m)\in [0,T]\times\R\times\Pc_1(\R_+),
 \ee*
with maximizer 
 $\widehat\beta(t,y,m,z):=\1_{\{zB(t,y,m)> 0\}}$. The value function can be characterized by means of the following backward SDE, which is the non-Markovian substitute of the HJB equation:
 \be*
 \mathrm{d}Y_t
 &\hspace{-2mm}=&\hspace{-2mm}
 Z_t \mathrm{d}X_t-H_t(X_t,\mu_t,Z_t)\mathrm{d}t,  ~t\in[0,T],~\mbox{and}~Y_T=U(X_T),~\mu-\mbox{a.s.}
 \\
 &\hspace{-4mm}=&\hspace{-4mm}
 \1_{\{ X_t > 0\}} \!\Big[
 F_t(X_t,Z_t)\mathrm{d}t
 +Z_t\Sigma(t,X_t,\mu_t)\mathrm{d}W^{\pi^*\!\!,\mu,\hat\beta}_t\Big],
 ~~
 F_t(x,z)
 :=
 \frac{ \int_{(0,\infty)}[zB^+\!-\!(zB)^+](t,y,\mu_t)\mu_t(\mathrm{d}y)}{1+\mu_t(\R_+^*)\1_{\{B(t,x,\mu_t)> 0\}}},
 \ee*
This backward SDE has a unique solution $(Y,Z)$ satisfying 
$\E\big[\sup_{t\le T}|Y_t|^2+\int_0^T|Z_t|^2\d t\big]<\infty$. Moreover, arguing as in \cite{touzidjete21}, we see that the nondecrease of the function $U$ implies that $Z\ge 0$. Then $F_t(X_t,Z_t)\equiv 0$, and we obtain $
 Y_t = \E^\P\big[U(X_T)|\Fc_t\big],$ $t\in [0, T].$
In order to complete the proof, we now show that $\E^\P[U(X_T)] = J_{\pi^*,\mu}(\pi^*) \ge J_{\pi^*,\mu}(\beta)$, for all $\beta\in\Ac$. Indeed, using the representation of $U(X_T)$ in terms of $Y$:
 \be*
 J_{\pi^*\!,\mu}(\beta)
 &=&
 \E^{\P_{\pi^*\!,\mu}^\beta}\big[ U(X_T) \big]
 \;=\;
 Y_0
 + \E^{\P_{\pi^*\!,\mu}^\beta}\bigg[\int_0^T Z_t \mathrm{d}X_t-H_t(X_t,\mu_t,Z_t) \mathrm{d}t \bigg].
 \ee*
Substituting the dynamics of $X$, it follows from appropriate localization that
  \be*
 J_{\pi^*,\mu}(\beta)
 &\hspace{-3mm}=&\hspace{-3mm}
 Y_0
 + \E^{\P_{\pi^*\!\!,\mu}^\beta}
   \!\Big[\int_0^{T}\!\!\!\Big(\1_{\{ X_t >0 \}}Z_t\frac{ b(t,X_t,\mu_t)
   \!+\!\int_{(0,\infty)}\beta(t,X_t,y)B(t,y,\mu_t)\mu_t(\mathrm{d}y)}
    {1+\mu_t(\R_+^*)\1_{\{B(t,X_t,\mu_t)> 0\}}}
    - H_t(X_t,\mu_t,Z_t)                     \Big)\mathrm{d}t                         \Big]
 \\
 &\hspace{-3mm}\le&
 Y_0 \;=\; \E^\P\big[U(X_T)\big],
 \ee*
where the last inequality follows from the definition of $H$. 
By the same argument, we see that the control process $\pi^*\in\Ac$ allows to reach the last upper bound $J_{\pi^*\!\!,\mu}(\pi^*)=\E^\P\big[U(X_T)\big]$.
\ep

\begin{remark}
Similar to \cite{touzidjete21}, we may use the equilibrium mutual holding strategy of Theorem \ref{thm:main} in order to build an approximate Nash equilibrium for the finite population mutual holding game. Due to the coupling by the hitting time of the origin, this would require nontrivial technical developments. We refrain from exploring this question in the present paper, and we leave it for future work. Instead, our focus in the next subsections is mainly on the evolution of the default probability which describes the systemic risk induced by the equilibrium connection between the agents. 
\end{remark}

\subsection{Autonomous characterization of the default probability}\label{ssec:characterization}

In general, the uniqueness of the solution to \eqref{eq:MFG-equilibriumSDE} may not hold. Nevertheless, we may always select a solution that admits an autonomous characterization of its default probability. Given a solution $X$ of \eqref{eq:MFG-equilibriumSDE}, we define $\mu_t:=\Lc(X_t)$ and we consider the map $c=(c_0,c_1):[0,T]\longrightarrow(0,1]\times\R_+$, defined by:
\beq \label{eq:c}
   c_{0}(t):=\mu_t\big[\R_+^*\big],
   ~~\mbox{and}~~
   c_{1}(t):=c_1(t,\mu_t),
\eeq 
where $c_1(t,\mu_t)$ is defined in  \eqref{eq:charac_c} and $1-c_0(t)$ is the probability that the representative agent defaults prior to time $t$. Alternatively, $1-c_0(t)=\mu[\tau\le t]$ is the cumulative distribution function (cdf) of the default time $\tau$ defined as the equity process hitting time of the origin.

Introduce further the coefficients $B^{\mu,c}, \Sigma^{\mu,c}: [0,T]\x \R \to \R$ with frozen distribution dependence by 
\be*
B^{\mu,c}(t,x)
:=
\frac{\big(b(t,x,\mu_t)+c_1({t})\big)^+}{1+c_0({t})}
-\big(b(t,x,\mu_t)+c_1({t})\big)^-,
&\mbox{and}&
\Sigma^{\mu,c}(t,x)
:=
\frac{\sigma(t,x,\mu_t)}{1+c_0({t})\1_{\{B^{\mu,c}(t,x)> 0\}}}.
\ee*
Notice that the above functions are all related to the given solution $X$ through its marginal laws $\mu_t$; this dependence is omitted for notation simplicity. 

\begin{theorem}\label{thm:more}
Let Assumption \ref{assum:bsigma} hold true. Then, there exist a weak solution $X$ to \eqref{eq:MFG-equilibriumSDE} and a stochastic process  $R^{t,x}$ satisfying the SDE  
\beq \label{eq:SDE}
R^{t,x}_t=x &\mbox{and}& \d R^{t,x}_u = B^{\mu,c}(u, R^{t,x}_u)\d u + \Sigma^{\mu,c}(u,R^{u,x}_u)\d W_u,\quad \forall u\in [t,T],
\eeq 
such that, denoting $\overline{B}^{\mu,c}(t,r):=\1_{\{r>0\}}\big(1,B^{\mu,c}(t,r)^+\big)$, we have
\beq\label{eq:c}
c(s)
&=&
\E\Big[\overline{B}^{\mu,c}\Big(s,R^{0,X_0}_s\Big)\Big]
- \int_{c_{0}(s)}^1\E\Big[\overline{B}^{\mu,c}\Big(s,R^{c_{0}^{-1}(u),0}_s\Big)\Big]\d u,\quad 
\mbox{for a.e.}~ s \in [0,T],
\eeq
where $c^{-1}_0$ is the left--continuous inverse of the non-increasing function $c_0$.
\end{theorem}

\begin{remark}
{\rm (i)} The proof of the last result actually provides an expression of the marginal distributions of the process $X$ in terms of those of the process $R$, see \eqref{eq:limit_equation}. However as the coefficients of the SDE defining the process $R$ depend on $\mu$, this representation is not explicit. 

\vspace{1mm}
\noindent {\rm (ii)} In the situation where the coefficients of the idiosyncratic risk process are independent of the distribution variable, i.e. $b\equiv b(t,x)$ and $\sigma\equiv \sigma(t,x)$, the characterization of Theorem  \ref{thm:more} is particularly useful. In this case, the coefficients $B^{\mu,c}$ and $\Sigma^{\mu,c}$ depend on $\mu$ only through $c$, and we denote in this case $(B^{\mu,c},\Sigma^{\mu,c})\equiv (B^{c},\Sigma^{c})$. Theorem \ref{thm:more} provides here an autonomous equation for the function $c$. In particular, the probability of default $1-c_0({t})$ does not require the full knowledge of $\Lc(X)$. See Subsection \ref{sec:numerics} for a numerical illustration. Moreover, once $c$ is determined, the equilibrium dynamics \eqref{eq:MFG-equilibriumSDE} would reduce to a standard SDE without mean field interaction in the case where  \eqref{eq:MFG-equilibriumSDE} admits a unique solution.
\end{remark}

When the drift $b$ has constant sign, recall from Remark \ref{rem:bconstantsign} that the coefficients $B$ and $\Sigma$  contain no singularity in $(t,x)$, and the solution of \eqref{eq:SDE} is unique. Together with the last characterization result, this observation allows to establish the uniqueness of the solution to \eqref{eq:MFG-equilibriumSDE} under additional conditions. 

\begin{theorem}\label{thm:uniqueness}
Let the conditions of Theorem \ref{thm:more} hold, and assume further that
\begin{itemize}
\item[\rm (i)] $b\equiv b(t,x)$ has constant sign, is differentiable in $x$ with bounded derivative, and H\"older-continuous in $t$;
\item[\rm (ii)] $\sigma\equiv \sigma(t,x)$ is twice differentiable in $x$ with bounded $\partial_x\sigma, \partial^2_{xx}\sigma$, and H\"older-continuous in $t$;
\item[\rm (iii)] $\rho(\d x)=p_0(x)\d x$, with continuous density $p_0$satisfying $\sup_{\lambda>0} \int_{0}^\infty \big(p_0(\lambda x)e^{-x^2} +x^{-2}p_0(x)\big)\d x<\infty$.
\end{itemize}
Then the mean field SDE \eqref{eq:MFG-equilibriumSDE} has a unique strong solution with corresponding maps $c$ uniquely defined by \eqref{eq:c}, and $c^0$ continuous and strictly decreasing.
\end{theorem}

\begin{remark}
A sufficient condition for $\sup_{\lambda>0} \int_{0}^\infty p_0(\lambda x)e^{-x^2}\d x<\infty$ is that $p_0(y)e^{-\frac{y^2}{\delta}}\underset{y\to\infty}{\longrightarrow} 0$, and $\int_0^\infty p_0(y)e^{-\frac{y^2}{\delta}}dy<\infty$. for some $\delta>0$. To see this, observe that
\begin{itemize}
\item $\sup_{\lambda\ge\sqrt{\delta}}\int_{0}^\infty p_0(\lambda x)e^{-x^2}\d x \le \sup_{\lambda\ge\sqrt{\delta}}\int_{0}^\infty p_0(\lambda x)\d x=\sup_{\lambda\ge\sqrt{\delta}}\frac1\lambda\rho[(0,\infty)]\le\frac1{\sqrt{\delta}}$,
\item Moreover, $\frac{\partial}{\partial\lambda}\int_{0}^\infty p_0(\lambda x)e^{-x^2}\d x  =\int_0^\infty x\dot p_0(\lambda x)e^{-x^2}\d x=-\frac1{2\lambda}\big[p_0(\lambda x)e^{-x^2}\big]_0^\infty+\frac1{2\lambda}\int_0^\infty p_0(\lambda x)e^{-x^2}\d x$, by direct integration by parts, so that our first condition implies that $\frac{\partial}{\partial\lambda}\int_{0}^\infty p_0(\lambda x)e^{-x^2}\d x >0$, and therefore $\sup_{\lambda\le\sqrt{\delta}}\int_{0}^\infty p_0(\lambda x)e^{-x^2}\d x = \int_{0}^\infty p_0(y)e^{\frac{-y^2}{\delta}}\frac{\d y}{\sqrt{\delta}}<\infty$ by our second condition.
\end{itemize}
\end{remark}
\subsection{Propagation of chaos}

As mentioned earlier, our strategy to prove Theorems \ref{thm:main} and \ref{thm:more} is to use an approximation argument. This section provides the concrete approximation which yields in particular a particle system approximation satisfying a result of propagation of chaos. Denoting by $H(x):=\1_{\{x>0\}}$ the Heaviside function, we introduce the following sequences:
\begin{itemize}
\item $(H^n)_{n\ge 1}\subset C^\infty(\R,[0,1])$ with ${\rm Supp}(H^n)\subset \R_+$ for all $n\ge 1$, and 
\be* 
H^n\longrightarrow H \mbox{ pointwisely on $\R$, and uniformly on } (-\infty,0]\cup [\eps,\infty) \mbox{ for all } \eps>0,
\ee*
for instance, we may take $H^n(x):=\1_{\{x>0\}}e^{-\frac{1}{nx}}$;
\item $(\rho^n)_{n\ge 1}\subset \Pc_q(\R_+)$ with ${\rm Supp}(\rho^n)\subset \R_+^* $, $\int_{\R_+} x^{-2}\rho^n(\mathrm{d}x)<\infty$, and converging to $\rho$ under $\Wc_q$,
one may consider for instance $\rho^n:=\Lc\big(\frac{|Z|}{n}+X_0\vee \frac1n\big)$, where $Z$ is a standard Gaussian random variable independent of $X_0$;

\item $(b^n,\sigma^n)_{n\ge 1}\subset C([0,T]\x\R\x\Pc_1(\R_+))$ with $(b^n,\sigma^n)\in C^\infty$ in $(t,x)$, Lipschitz in $m$, and converging to $(b,\sigma)$, uniformly in $(x,m)$ and a.e. in $t$.
\end{itemize}
For all $n \ge 1,$ we may define by Lemma \ref{lemma:existence_c}, the map $c^{n}_1:[0,T] \x \Pc_1(\R_+) \to \R_+$ as the solution of
\be*
c^{n}_1(t,m)=\frac{1}{1+m(H^n)} \int_{\R} \big(c^{n}_1(t,m)+b^n(t,x,m)\big)^+ H^n(x) m(\mathrm{d}x),
~\mbox{for all}~(t,m)\in [0,T] \x \Pc_1(\R_+),
\ee*
and we introduce the functions $B^n,\Sigma^n:[0,T] \x \R \x \Pc_1(\R_+) \longrightarrow \R$ by
\begin{equation}\label{BnSigman}
    B^n(\cdot,m)
    :=
    \frac{(b^n+c^{n}_1)^+(\cdot,m)}{1+m(H^n)} - (b^n+c^{n}_1)^-(\cdot,m),
    ~\mbox{and}~
    \Sigma^n(\cdot,m):=\frac{\sigma^n(\cdot,m)}{1+ m(H^n) H^n(B^n(\cdot,m))}.
\end{equation}

\begin{theorem} \label{thm:approx_regularity}
Let Assumption \ref{assum:bsigma} hold.

\vspace{1mm}
  
 \noindent {\rm (i)} For every $n \ge 1$, there exists a unique strong solution $(Y^n_t)_{t \in [0,T]}$ satisfying $\Lc(Y^n_0)=\rho^n$ and 
    $$
        \mathrm{d}Y^n_t
        \!=\!
        B^n\big(t,Y^n_t,\mub^n_t\big) \mathrm{d}t
        +
        \Sigma^n\big(t,Y^n_t,\mub^n_t\big)\mathrm{d}W_t, \mbox{ with }
        \mub^n_t \!:=\! \Lc\big(\overline{Y}^n_t\big),
        \mbox{ and }
        \overline{Y}^n_t \!:=\! Y^n_t H^n\big(I^n_t\big),
        ~I^n_t \!:=\inf_{0\le s\le t} Y^n_s.
    $$
    {\rm (ii)} We have $(H^n-H)(I^n_t)\longrightarrow 0$ in $\L^1$, for all $t \in [0,T]$;
    
    \vspace{1mm}
    
\noindent {\rm (iii)} Denoting $X^n:=Y^nH\big(I^n\big)$ and $\mu^n:=\Lc\big(X^n\big)$, the sequence $(\mu^n)_{n \ge 1} \subset \Pc\big(C([0,T])\big)$ is relatively compact in $\Wc_1$. 
    Furthermore, the limit $\mu$ of any convergent subsequence $({\mu}^{n_k})_k$ is a weak solution of \eqref{eq:MFG-equilibriumSDE}, and satisfies:
      \be*
    c^{n_k}_0(t)
    :=
    \mu^{n_k}_t(H^{n_k})
    \underset{k \to \infty} {\longrightarrow}
    c_0(t)
     :=
     \mu_t(\R_+^*) 
     &\mbox{and}& 
     c^{n_k}_1(t,\mu^{n_k}_t)\underset{k \to \infty} {\longrightarrow}
    c_1(t,\mu_t),
    ~\mbox{for all}~
    t \in [0,T].
    \ee*
\end{theorem}
\begin{remark}\label{rem:choice}
{\rm (i)} As previously mentioned, we do not have a uniqueness result for the SDE \eqref{eq:MFG-equilibriumSDE}, and choosing different approximating coefficients may lead to different limits.

\medskip
\noindent {\rm(ii)} The convergence result of \Cref{thm:approx_regularity} (ii) implies that
$
    \lim_{n \to \infty} \Wc_q \left(\mub^n_t,\mu^n_t \right)=0,
$ for all $t \in [0,T].$
\end{remark}

We next introduce a particle system approximation which induces in particular an approximation method for the survival probability $c^0$. For all $N \ge 1$, let $(Y^{n,N,1},\cdots,Y^{n,N,N})$ be the process defined by the SDE  
\be* 
        &\mathrm{d}Y^{n,N,i}_t
        =
        B^n(t,Y^{n,N,i}_t,\mub^{n,N}_t) \mathrm{d}t
        +
        \Sigma^n(t,Y^{n,N,i}_t,\mub^{n,N}_t)\mathrm{d}W^i_t,&
        \\
        &\displaystyle \mbox{with }
        \mub^{n,N}_t:=\frac{1}{N} \sum_{i=1}^N\delta_{\overline{Y}^{n,N,i}_t},~ \overline{Y}^{n,N,i}_t:=Y^{n,N,i}_tH^n\big(I^{n,N,i}_t\big), ~
I^{n,N,i}_t:=\inf_{0\le s\le t} Y^{n,N,i}_s,&
\ee*
and started from any initial data satisfying $\overline{\mu}^{n,N}_0\longrightarrow\rho_n$, as $N\to\infty$, in $\Wc_q$.
The following propagation of chaos result is a direct consequence of \cite[Proposition 4.15]{djete2019general}. 

\begin{proposition} \label{prop:approx_poc}
Let $X^{n,N,i}:=Y_nH(I^{n,N,i})$. Then, under Assumption \ref{assum:bsigma}, we have $\frac1N\sum_{i=1}^N\delta_{X^{n,N,i}}\longrightarrow\mu^{n}$, as $N \to \infty$, in $\Wc_1$, for all $n\ge 1$.
    \end{proposition}
    
\subsection{A numerical example}
\label{sec:numerics}

In order to illustrate the effect of mutual holding, we end Section \ref{sec:main} by the following example where the Ornstein--Uhlenbeck SDE is used to model the idiosyncratic risk process. Namely, for $t\ge 0$
\be* 
\d \tilde X_t = \1_{\{\tilde X_t>0\}} \Big (( \lambda-\tilde X_t) \d t +  \d W_t\Big),
~\mbox{and}~
\d X_t = \1_{\{X_t>0\}} \Big (B^c(t,X_t) \d t +\Sigma^c(t,X_t)  \d W_t\Big),
\ee* 
where $\tilde X$ and $X$ stand for the dynamics of a representative agent without and with the mutual holding. In particular, the new drift and volatility functions $B^c, \Sigma^c :\R_+\times \R \to \R$ are defined by
\be*
B^{c}(t,x)
:=
\frac{\big(\lambda-x+c_1(t)\big)^+}{1+c_0({t})}
-\big(\lambda-x+c_1({t})\big)^- 
&\mbox{and}&
\Sigma^{c}(t,x)
:=
\frac{1}{1+c_0({t})\1_{\{x<\lambda+c_1(t)\}}},
\ee*
where, in view of \eqref{pn},  $c=(c_0,c_1)$ is a fixed point of the map $\overline \Lambda : c\mapsto \overline \Lambda[c]:=(\overline \Lambda_0[c], \overline \Lambda_1[c])$ defined by
\be* 
\overline \Lambda_0[c](s) &:=& \P\big[ R_{s}^{0,Z}>0\big] +\int_{0}^{s} \P\big[ R_{s}^{t,0}>0\big]\d c_0(t) \\
\overline \Lambda_1[c](s) &:=& \E\left[ \1_{\{R_{s}^{0,Z}>0\}}B^c({s},R_{s}^{0,Z})^+\right] +\int_{0}^{s}\E\left[ \1_{\{R_s^{{t},0}>0\}}B^c({s},R_{s}^{t,0})^+\right]\d c_0(t),
\ee* 
and $R^{t,x}$ is defined by the stochastic differential equation
\be*
R^{t,x}_s=x + \int_t^s B^{c}(u, R^{t,x}_u)\d u + \int_t^s \Sigma^{c}(u,R^{t,x}_u)\d W_u,\quad \forall s\ge t.
\ee*
We examine the evolution of the default probabilities $\tilde D(t):=\P[\tilde X_t\le 0]$ and $D(t):=\P[ X_t\le 0]$. Clearly,
\be*
 \tilde D(\infty)=1-\P\big[\tilde X_0+\lambda (e^{t}-1)+W_{e^{2t-1}}>0, ~ \forall t\ge 0\big]=1.
\ee*
By using similar arguments as in \Cref{lem:ito-martingale} (ii), we may also prove that $D(\infty)=1$.  We fix a time horizon $T>0$ and adopt the iteration $c^{k+1}:=\overline \Lambda[c^k]$ for all $k\ge 0$ with some initial guess $c^0$. To do so, we approximate $\overline \Lambda[c]$ by time discretization and Monte Carlo simulation, i.e. taking $M,N\in\N$, $\Delta t:=T/N$, $t_n:=n\Delta t$ and independent $Z_1,\ldots, Z_M$ random variables distributed according to some $\rho$, one has
\be* 
\overline \Lambda_0[c](t_n) 
&\!\!\!\!\approx &\!\!\!\!
 \P\big[ R_{t_n}^{0,Z}>0\big] + \sum_{k=0}^{n-1} \P\big[ R_{t_n}^{t_k,0}>0\big]\big(c_0(t_{k+1})-c_0(t_k)\big)\\
&\!\!\!\!\approx&\!\!\!\! \frac{1}{M}\sum_{m=1}^M \Big[ \1_{\{R_{m, t_n}^{0,Z_m}>0\}}+\sum_{k=0}^{n-1}  \1_{\{R_{m, t_n}^{t_{k},0}>0\}}\big(c_0(t_{k+1})-c_0(t_k)\big)\Big]\\
\overline \Lambda_1[c](t_n) 
&\!\!\!\!\approx &\!\!\!\! \E\left[ \1_{\{R_{t_n}^{0,Z}>0\}}B^c({t_n},R_{t_n}^{0,Z})^+\right] + \sum_{k=0}^{n-1} \E\left[ \1_{\{R_{t_n}^{t_k,0}>0\}}B^c({t_n},R_{t_n}^{t_k,0})^+\right]\big(c_0(t_{k+1})-c_0(t_k)\big)\\
&\!\!\!\!\approx& \!\!\!\!\frac{1}{M}\sum_{m=1}^M \Big[ \1_{\{R_{m, t_n}^{0,Z_m}>0\}}B^c\big(t_n,R_{m, t_n}^{0,Z_m}\big)^++\sum_{k=0}^{n-1}\1_{\{R_{m, t_n}^{t_{k},0}>0\}}B^c\big(t_n,R_{m, t_n}^{t_{k},0}\big)^+\big(c_0(t_{k+1})-c_0(t_k)\big)\Big], 
\ee* 
where $R^{t_k,x}_{m,t_k}:=x$ and $R^{t_k,x}_{m,t_{i+1}}:=R^{t_k,x}_{m,t_i} +  \Delta t B^{c}(t_i, R^{t_k,x}_{m,t_i}) + \sqrt{\Delta t} \Sigma^{c}(t_i,R^{t_k,x}_{m,t_i})G_{m,i+1}$ for  $i=k,\ldots, N-1$, and 
 $(G_{m,n}: 1\le m\le M, 1\le n\le N)$ are $MN$ independent standard Gaussian random variables. Notice that we are ignoring here the major difficulty related to the discontinuity of $\Sigma^c$, as this is not the main concern of the present paper.

With $T=10$ and $\rho(\d x)= \1_{\{x>0\}}e^{-x}\d x$, we illustrate in Figure \ref{iter} the convergence of our iteration. 
\begin{figure}[!t]
\centering
\includegraphics[width=0.6\textwidth]{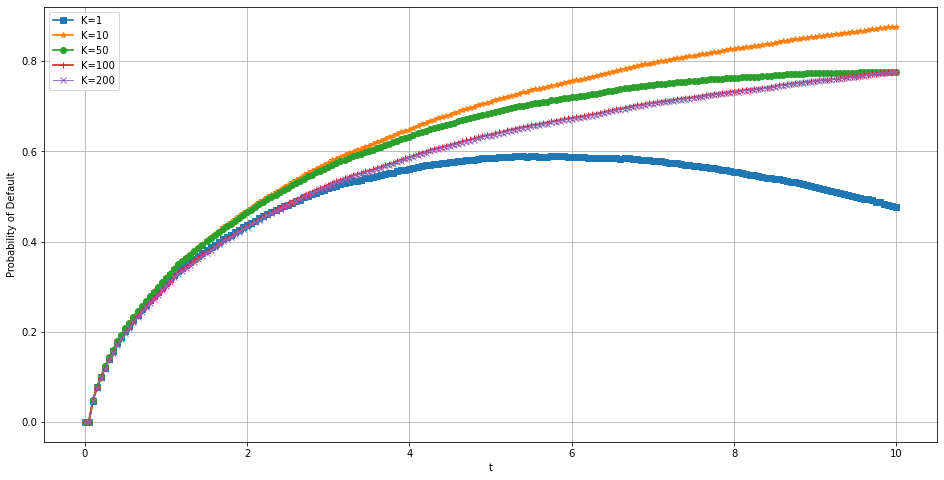}
\vspace{-5mm}
\caption{$\lambda=1$, $M=10000$, $N=200$, $k\in \{1,10, 50, 100, 200\}$}
\label{iter}
\end{figure}
Next, with different parameters $\lambda$, we see in Figures \ref{com1}  and \ref{com2}  that the mutual holding significantly decreases the propagation of defaults, and that a larger mean reversion level induces a more significant effect of mutual holding on the default probability. In other words a larger mean reversion level leads to an equilibrium equity process with lower systemic risk.  

\begin{figure}[!t]
\centering
\includegraphics[width=0.6\textwidth]{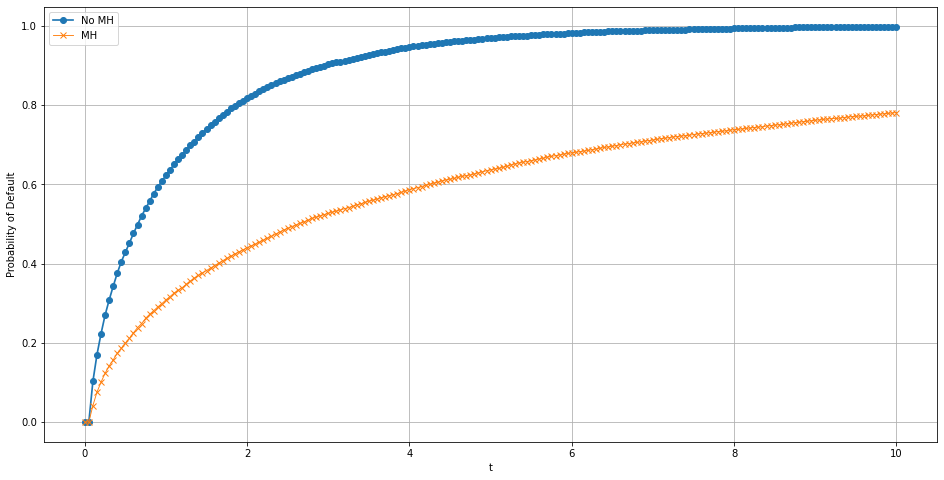}
\vspace{-5mm}
\caption{$\lambda=0.8$, $M=10000$, $N=200$, $k=200$}
\label{com1}
\end{figure}
\begin{figure}[!t]
\centering
\includegraphics[width=0.6\textwidth]{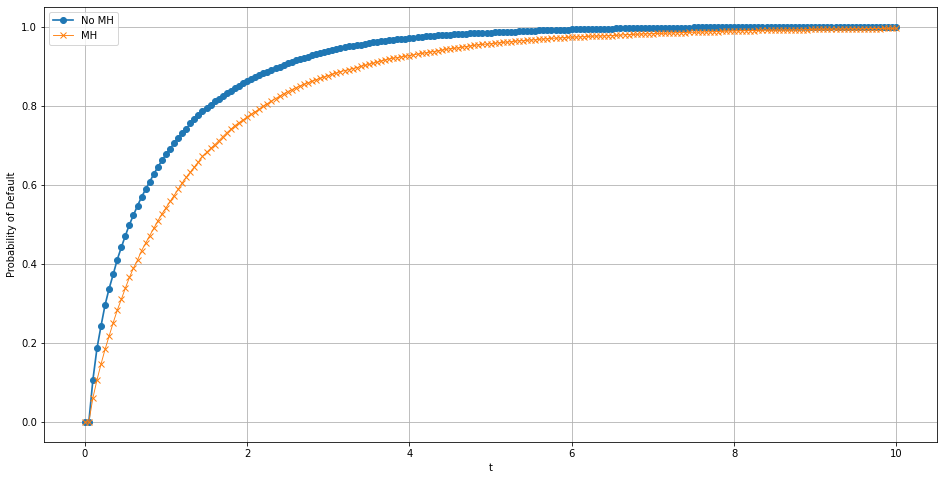}
\vspace{-5mm}
\caption{$\lambda=0.1$, $M=10000$, $N=200$, $k=200$}
\label{com2}
\end{figure}

\section{Particle system approximation} \label{sec:proof1}

This section is devoted to the proof of Theorems \ref{thm:main} and \ref{thm:approx_regularity}. Note that Theorem \ref{thm:main} (i) is an immediate consequence of Theorem  \ref{thm:approx_regularity} (iii). 

\subsection{Proof of Theorem  \ref{thm:approx_regularity}}
   
(i) Let us introduce the map $F:  \Pc\big(C([0,T])\big) \longrightarrow \Pc(\R_+)$ defined by
    \be*
        F(\nu)(\varphi)
        :=
        \int_{C([0,T])} \varphi\Big(\mathbf{y}_t H^n\big(\inf_{0\le s \le t} \mathbf{y}_s\big)\Big) \nu(\mathrm{d}\mathbf{y}),\quad \mbox{for all } \varphi\in C_b(\R_+). 
    \ee*
Using $F$, we can rewrite the dynamics of $Y^n$ as
\be*
        \mathrm{d}Y^n_t
        =
        \overline{B}^n\big(t,Y^n_t,\Lc(Y^n_{t \wedge \cdot})\big) \mathrm{d}t
        +
        \overline{\Sigma}^n\big(t,Y^n_t,\Lc(Y^n_{t \wedge \cdot})\big)\mathrm{d}W_t,
        ~\mbox{with}~
        (\overline{B}^n,\overline{\Sigma}^n)(t,x,\nu):=(B^n,\Sigma^n)(t,x,F(\nu)).
    \ee*
As the map  $C([0,T])\ni \mathbf{y} \mapsto \inf_{0\le s \le t} \mathbf{y}_s \in \R$ is Lipschitz, it follows that $F$ is $\Wc_1-$Lipschitz. Moreover, by Lemma \ref{lemma:existence_c}, the coefficients $\overline{B}^n,\overline{\Sigma}^n :[0,T] \x \R \x \Pc_q(C([0,T])) \to \R$ are Lipschitz in $(x,\nu)$ uniformly in $t\in [0,T]$. Then, the existence and uniqueness of $Y^n$ follow  from the path dependent extension of standard results, see e.g. \cite[Theorem A.3]{djete2019mckean} 

\medskip
\vspace{1mm}  

\noindent (ii) Denote $\delta^n\!H:=H^n-H$. After possibly passing to a subsequence, we may assume without loss of generality that the bounded sequence $(\sup_{t \in [0,T]} \E\big[|\delta^n\!H(I^n_t)| \big])_{n \ge 1}$ has a limit. In order to prove the required result, we now show that this limit is zero.
Let $(t_{n})_{n \ge 1} \subset [0,T]$ be such that 
$$
    \sup_{t \in [0,T]} \E\big[|\delta^n\!H(I^{n}_t)| \big]\le \E\big[|\delta^n\!H(I^{n}_{t_n})| \big] + 2^{-n},
$$
and observe that for all $\varepsilon>0,$ we have 
    \begin{equation} \label{eq:conv_H}
       \E\big[\delta^n\!H(I^n_{t_n})| \big]
        =
        \E\Big[|\delta^n\!H(I^n_{t_n})|\1_{\{|I^n_{t_n}| \ge \varepsilon\}} \Big]
        \!+\!
        \E\Big[|\delta^n\!H(I^n_{t_n})|\1_{\{|I^n_{t_n}| < \varepsilon\}}  \Big] 
       \le 
        \sup_{|x| \ge \varepsilon} |\delta^n\!H(x)|
        \!+\!
        2\P[|I^n_{t_n}| \le \varepsilon].
    \end{equation}
Since $(\rho^n=\Lc(Y^n_0))_{n \ge 1} \subset \Pc_q(\R_+)$ converges under $\Wc_q$ with $q>1$, the sequence $(\mu^{n})_{n \ge 1}=(\Lc(Y^{n}))_{n \ge 1}$ is relatively compact in $\Wc_1$. Then, after possibly passing to a subsequence $\mu^{n}\longrightarrow\mu^\infty$ in $\Wc_1$. Using \cite[Proposition 9.1]{touzidjete21}, we see that $\mu^\infty$ is the distribution of $Y^\infty$ that solves an SDE of non--degenerate diffusion coefficient. Up to a subsequence again, we can assume that $\lim_{n \to \infty} t_{n}=t_\infty$ for some $t_\infty \in [0,T]$.
Then, $\P \big[ I^\infty_{t_\infty}=a \big]=0$ for any $a \in \R$ where $I^\infty_t:= \inf_{0 \le s \le t} Y^\infty_s$. By the Portmanteau Theorem and the Lebesgue dominated convergence Theorem, we have
\begin{eqnarray*}
    \lim_{n \to \infty}\P[|I^{n}_{t_{n}}| \le \varepsilon]
     =
     \P[|I^{\infty}_{t_{\infty}}| \le \varepsilon]
     &\mbox{and}&
     \Lim_{\varepsilon \to 0} \P[|I^{\infty}_{t_{\infty}}| \le \varepsilon]
    =
    \P[|I^{\infty}_{t_{\infty}}| =0]=0.
\end{eqnarray*}
Then, for any limit $\mu^\infty=\Lc(Y^\infty)$ of the convergent (sub--)sequence $\mu^n=\Lc(Y^n)$, it follows from \eqref{eq:conv_H}, together with the convergence of $(H^n)_{n \ge 1}$, that $\Lim_{n \to \infty}\E\big[|\delta^{n}\!H(I^{n}_{t_{n}})| \big]=0$.

\medskip

\vspace{1mm}  

\noindent (iii) 
We recall that $X^n:=Y^n H(I^n)$, so that $X^n_t=Y^n_{t \wedge \tau^n}$, where $\tau^n:=\inf\{ t \le T,\;Y^n_t \le 0\}$. Using this SDE representation of $X^n$, we deduce that $(\Lc(X^n ))_{n \ge 1}$ is relatively compact in $\Wc_1$. Since $\lim_n\sup_{t \in [0,T]} \E\big[|\delta^{n}\!H^n(I^n_t)| \big]=0$, we deduce that
\begin{align*}
    \Lim_{n \to \infty}\Wc_1 \big( \big(\Lc(X^n_t )\big)_{t \in [0,T]}, \big(\Lc(\overline{Y}^n_t)\big)_{t \in [0,T]} \big)=0.
\end{align*}
Then, after possibly passing to a subsequence, we may assume that $\Lc(Y^n)=\mu^{n}\longrightarrow\mu^\infty=\Lc(Y^\infty)$ in $\Wc_1$. Denote $\mub^\infty_t:=\Lc\big(\overline Y^\infty_t \big)$ with $\overline Y^\infty_t:=Y^\infty_t H(I^\infty_t)$ and $I^\infty_t:=\inf_{s \le t} Y_s$ . For all Lipschitz function $f:\R^3 \to \R$,
    \begin{align*}
    \limsup_{n \to \infty} 
    |\E [f (Y^n_t,H^n(I^n_t),I^n_t )]
      &-\E [f (Y^\infty_t, H(I^\infty_t),I^\infty_t )]| 
   \\
       & \le
        \limsup_{n \to \infty} |\E [f (Y^n_t, H(I^n_t),I^n_t )]-\E[f (Y^\infty_t, H(I^\infty_t),I^\infty_t )]|=0.
    \end{align*}
Therefore, 
    $$
        \lim_{n \to \infty} \Lc\big( Y^n_t, H^n(I^n_t),I^n_t \big)
        =
        \Lc\big( Y^\infty_t, H(I^\infty_t),I^\infty_t \big) \mbox{ in } \Wc_1.
    $$
In particular,  $\mub^n_t=\Lc(Y^n_t H^n(I^n_t))\longrightarrow\Lc(Y^\infty_t H(I^\infty_t))=\mub^\infty_t$ in $\Wc_1$ as $n\to\infty$.

\medskip

\vspace{1mm}  

\noindent (iv) We next prove the convergence of $c^{n}_0(t)=\mub^n_t(H^n)$. For $\varepsilon>0$, one has
    \begin{align*}
        | \mub^n_t(H^n) - \mub^n_t(H)|
        &\le \E\Big[|H^n(\overline Y^n_t)-H(\overline Y^n_t)|\1_{\{|\overline Y^n_t| \ge \varepsilon \}} \Big]
        +
        \E\Big[|H^n(\overline Y^n_t)-H(\overline  Y^n_t)|\1_{\{|\overline Y^n_t| \le \varepsilon \}}\Big]
        \\
        &\le 
        \sup_{|x| \ge \varepsilon} |H^n(x)-H(x)| +
        \E\Big[|H^n(\overline Y^n_t)-H(\overline  Y^n_t)|\1_{\{\overline Y^n_t \le \varepsilon \}} \1_{\{I^n_t>0 \}}\Big]
        \\
        &\le 
        \sup_{|x| \ge \varepsilon} |H^n(x)-H(x)| +
        \P \big[Y^n_tH^n(I^n_t) \le \varepsilon,I^n_t>0 \big]     \underset{n\to\infty}\longrightarrow
        \P \big[Y^\infty_tH(I^\infty_t) \le \varepsilon, I^\infty_t>0 \big]
    \end{align*}
 where the last convergence follows from the Portmanteau Theorem and $\P[Y^\infty_t=\varepsilon] + \P[I_t^\infty=0]=0$. Applying the dominated convergence theorem, one has $\lim_{\eps\to 0}\P \big[Y^\infty_tH(I^\infty_t) \le \varepsilon, I^\infty_t>0 \big]\le \P \big[I^\infty_t \le0, I^\infty_t>0 \big]=0$. Using again the Portmanteau Theorem  and $\P[  Y^\infty_t=0]+\P[I_t^\infty=0]=0$, it follows that
    \be*
        \lim_{n \to \infty}\mub^n_t(H)
        =
        \lim_{n \to \infty}\P[\overline Y^n_t>0]
        =
        \lim_{n \to \infty}\P[ Y^n_t>0, I^n_t >0]
        =
        \P[ Y^\infty_t>0, I^\infty_t >0]
        =
        \P[\overline Y^\infty_t>0]
        =
        \mub^\infty_t(H).
    \ee*
Finally, combining the obtained convergence results, we find
    \be*
        \lim_{n \to \infty}| \mub^n_t(H^n) - \mub^\infty_t(H)  |
        \le
        \lim_{n \to \infty} | \mub^n_t(H^n) - \mub^n_t(H)  |
        +
        \lim_{n \to \infty} | \mub^n_t(H) - \mub^\infty_t(H)  |
        =
        0.
    \ee* 
(v) To prove the convergence of $c^{n}_1(t,\mub^n_t)$, we compute 
    \be*
    &&\big|c^{n}_1(t,\mub^n_{t}) - c_{1}(t,\mub^\infty_t) \big|\\
    &=&
    \bigg|  \int_{\R_+} \frac{(c^{n}_1(t,\mub^n_{t}) + b^n(t,x,\mub^n_{t}))^{+} H^n(x)}
                                       {1+\mub^n_t(H^n)}\mub^n_{t}(\d x)
    - \int_{\R_+} \frac{(c_1(t,\mub^\infty_{t}) + b(t,x,\mub^\infty_{t}))^{+} H(x) }
                               {1+\mub^\infty_t(H)}\mub^\infty_{t}(\d x)
    \bigg|
    \\
    &\le& \big|c^{n}_1(t,\mub^n_{t}) - c_1(t,\mub^\infty_t) \big| \frac{\mub^n_t(H^n)}{1+\mub^n_t(H^n)}
    + \alpha^n + \beta^n +\gamma^n,
\ee*
where
\be*   
\alpha^n &:=& \int_{\R_+} \big | b^n(t,x,\mub^n_t) - b(t,x,\mub^\infty_t)\big| H^n(x)\mub^n_{t}(\d x) \\
\beta^n &:=&\left |\int_{\R_+} \big( c_1(t,\mub^\infty_{t}) + b(t,x,\mub^\infty_{t}) \big)^{+} \big(H^n(x)\mub^n_{t}(\d x)-H^n(x)\mub^\infty_{t}(\d x)\big) \right |\\
\gamma^n &:=&\big|\mub^n_t(H^n)- \mub^\infty_t(H)\big|\int_{\R_+} \big( c_1(t,\mub^\infty_{t}) + b(t,x,\mub^\infty_{t}) \big)^{+} H(x)\mub^\infty_{t}(\d x). 
\ee*
By convergence assumption of $b^n$, $\lim_{n \to \infty} \alpha^n =  0$ holds. Moreover $\lim_{n \to \infty} \gamma^n =  0$ by the previous step. Using similar arguments as those developed for the convergence $\lim_{n \to \infty}\big|\mub^n_t(H^n)- \mub^\infty_t(H)\big|=0$, we obtain
\be*
\beta^n =\left| \E\big[\big(c_1(t,\mub^\infty_{t}) + b(t,\overline Y^n_t,\mub^\infty_{t})\big)^{+}H^n(\overline Y^n_t)-H^n(\overline Y^\infty_t)\big]\right|\underset{n \to \infty} {\longrightarrow} 0.
\ee*
Hence,
$\lim_{n \to \infty} \big|c^{n}_1(t,\mub^n_{t}) - c_1(t,\mub^\infty_t) \big|
    \le    2\lim_{n \to \infty} (\alpha^n+\beta^n+\gamma^n)
    = 0$.

\medskip

\vspace{1mm}
    
\noindent (vi) We now have all the ingredients for the convergence of $(B^n,\Sigma^n)$. Passing to the limit in the equation satisfied by $B^n$, it follows from the convergence of $b^n$, $c^{n}_0$ and $c^{n}_1$ that $\lim_{n \to \infty}B^n(t,x,\mub^n_t)=B(t,x,\mub^\infty_t)$ for all $x \in \R$. Let $G\in C_b(\R^2)$ be Lipschitz. Notice that
    \begin{align*}
        &\E\big[\big|G\big( Y^\infty_t,  H^n(B^n(t,Y^\infty_t,\mub^n_t))  \big) - G\big( Y^\infty_t,  H(B^n(t,Y^\infty_t,\mub^n_t))  \big) \big| \big]
        \\
        &\le 
        \E\big[\big|G\big( Y^\infty_t,  H^n(B^n(t,Y^\infty_t,\mub^n_t))  \big) - G\big( Y^\infty_t, H(B^n(t,Y^\infty_t,\mub^n_t))  \big) \big|\1_{\{|B^n(t,Y^\infty_t,\mub^n_t)| \ge \varepsilon\}} \big] 
        \\
        &\hspace{5mm}+
        C\P [|B^n(t,Y^\infty_t,\mub^n_t)| \le \varepsilon]
        \\
        &\le
        C \sup_{|x| \ge \varepsilon} |H^n(x)- H(x)| +
        C\P [|B^n(t,Y^\infty_t,\mub^n_t)| \le \varepsilon].
    \end{align*}
As $\mu^\infty$ is the distribution of $Y^\infty$ that solves an SDE of non--degenerate diffusion coefficient, it follows that $\mu^\infty_t(\mathrm{d}y)\mathrm{d}t$ admits a density w.r.t. the Lebesgue measure over $\R \x [0,T]$. Therefore, thanks to the assumptions satisfied by $b$ in \Cref{assum:bsigma}, we can check that $\int_0^T\P [|B(t,Y^\infty_t,\mub^\infty_t)| = \varepsilon] \mathrm{d}t=\int_0^T\int_{\R} \1_{\{|B(t,y,\mub^\infty_t)| = \varepsilon \}} \mu^\infty_t(\mathrm{d}y) \mathrm{d}t=0$. 
   Since $\lim_{n \to \infty}B^n(t,x,\mub^n_t)=B(t,x,\mub^\infty_t)$ for all $x\in \R$, we can deduce by the Portemanteau theorem that
    \be*
        \lim_{n \to \infty}\P [|B^n(t,Y^\infty_t,\mub^n_t)| \le \varepsilon]
        =
        \P [|B(t,Y^\infty_t,\mub^\infty_t)| \le \varepsilon],\;\;\mbox{for a.e. } t \in [0,T].
    \ee*
    With similar techniques to those used previously, we get by taking $n \to \infty$ and $\varepsilon \to 0$, for a.e. $t \in [0,T]$,
    \be*
        \lim_{n \to \infty} \E\big[\big|G\big( Y^\infty_t, H^n(B^n(t,Y^\infty_t,\mub^n_t))  \big) - G\big( Y^\infty_t, H(B^n(t,Y^\infty_t,\mub^n_t))  \big) \big| \big]\le  \lim_{\eps \to 0}\P [|B(t,Y^\infty_t,\mub^\infty_t)| \le \varepsilon] = 0,
    \ee*
    where the last equality is ensured by Assumption \ref{assum:bsigma} (ii). Again by the Portemanteau theorem, one has for a.e. $t \in [0,T]$, 
\be*    
    \lim_{n \to \infty} \E\big[G\big( Y^\infty_t,  H(B^n(t,Y^\infty_t,\mub^n_t))  \big) \big| \big]=\E\big[G\big( Y^\infty_t,  H(B(t,Y^\infty_t,\mub^\infty_t))  \big) \big| \big],
    \ee* 
which implies that for a.e. $t \in [0,T]$
    \be*
        \lim_{n \to \infty} \Lc\big( Y^\infty_t,  H^n(B^n(t,Y^\infty_t,\mub^n_t))  \big)
        =
        \Lc\big( Y^\infty_t, H(B(t,Y^\infty_t,\mub^\infty_t))  \big)
        ~~\mbox{in}~~\Wc_1.
    \ee*
    As $\lim_{n \to \infty}\mub^n_t(H^n)=\mub^\infty_t(H)$ and $\lim_{n \to \infty} \mub^n_t=\mub^\infty_t$, using the previous result, we may deduce that 
    \begin{align*}
        &\E \big[|\Sigma^n(t,Y^\infty_t,\mub^n_t)
                    -\Sigma(t,Y^\infty_t,\mub^\infty_t)| \big] 
\\ &=
        \E \left[\left|\frac{\sigma^n(t,Y^\infty_t,\mub^n_t)}{1+ \mub^n_t(H^n)  H^n(B^n(t,Y^\infty_t,\mub^n_t))}-\frac{\sigma(t,Y^\infty_t,\mub^\infty_t)}{1+ \mub^\infty_t(H)  H(B(t,Y^\infty_t,\mub^\infty_t))}\right| \right]
  \;\longrightarrow\; 0,
  \mbox{ for a.e. }t \in [0,T].
    \end{align*}
   On the other hand, since $\mu^\infty_t(\mathrm{d}x)\mathrm{d}t$ has a density with respect to the  Lebesgue measure on $[0,T] \x \R$, then it must hold that $\lim_{n \to \infty} \Sigma^n(t,x,\mub^n_t)=\Sigma(t,x,\mub^\infty_t)$ a.e. $(t,x) \in [0,T] \x \R$.     

\medskip
    
    \vspace{1mm}
    
\noindent (vii) Given the dynamics of $Y^n$ in the statement of \Cref{thm:approx_regularity}, it follows from the convergence results established in the previous steps that  
    \be*
        Y^\infty_t
        =Y^\infty_0+
        \int_0^t B(s,Y^\infty_s,\mub^\infty_s) \mathrm{d}s
        +
        \int_0^t\Sigma(s,Y^\infty_s,\mub^\infty_s) \mathrm{d}W_s,\quad \forall t\in [0,T].
    \ee*
    Therefore, by noticing $Y^\infty_tH(I^\infty_t)=Y^\infty_{t \wedge \tau^\infty}$ and
   $\mub^\infty_t(H)
        = \P [Y^\infty_tH(I^\infty_t) > 0]
        =        \P[Y^\infty_{t \wedge \tau^\infty}>0]$, 
    where $\tau^\infty:=\inf\{t \in [0,T]: Y^\infty_t \le 0\}$, we have showed that $(c^n_0,c^n_1)(t) \underset{n \to \infty} {\longrightarrow} (\mub^\infty_t(H),c_1(t,\mub^\infty_t))$. It is straightforward that $Y^\infty_{ \wedge \tau^\infty}$ satisfies \eqref{eq:MFG-equilibriumSDE}. It remains to prove that 
    \begin{align*}
        \Lim_{n \to \infty}\Wc_1 \big( \Lc(X^n_\cdot ), \Lc({Y}^\infty_{\cdot \wedge \tau^\infty}) \big)=0.
    \end{align*}
To see this, let $0 \le r_1 \le \cdots \le r_L$ and $0 \le t_1 \le \cdots \le t_L \le T$, and observe that 
    $
        \P \big[ \cap_{i=1}^L \{ X^n_{t_i} \ge r_i \} \big]
        =
        \P \big[ \cap_{i=1}^L \{ Y^n_{t_i} \ge r_i, I^n_{t_i} >0 \} \big]
        \underset{n\to\infty}{\longrightarrow}
        \P \big[ \cap_{i=1}^L \{ Y^\infty_{t_i} \ge r_i, I^\infty_{t_i} >0 \} \big]
        =
        \P \big[ \cap_{i=1}^L \{ Y^\infty_{t_i \wedge \tau^\infty} \ge r_i \} \big],
    $
by Portemanteau Theorem, since $\sum_{i=1}^L \P[Y^\infty_{t_i}=r_i]=0$. 
As $(\Lc(X^n))_{n \ge 1}$ converges (along the subsequence), this allows to identify its limit to $\Lc({Y}^\infty_{\cdot \wedge \tau^\infty})$.
    \ep


\subsection{Derivation of the autonomous characterization of the default probability}
\label{sec:autonomous}

This section is devoted to the proof of Theorem \ref{thm:more}. Unless otherwise specified, we use throughout the approximation arising in Theorem \ref{thm:approx_regularity}, see also \Cref{rem:choice} (ii). Notice that, the process $X^n:=Y^n H\big(I^n\big)$, with $I^n$ the running minimum of the process $Y^n$ satisfies
\be*
   \mathrm{d}X^n_t = \1_{\{X^n_t>0 \}} \big[ B^n(t,X^n_t,\mub^n_t)\d t +\Sigma^n(t,X^n_t,\mub^n_t)\d W_t\big],\quad \mbox{with } \mub^n_t=\Lc(Y^n_t H^n(I^n_t))  \mbox{ and } \Lc(X_0)=\rho^n,
\ee*
and recall the notations $\mu^n_t:=\Lc(X^n_t)$ and $c^n_0(t):=\mu^n_t[\R_+^*]$. 

\begin{lemma}\label{lem:decom}
Let the conditions of Theorem \ref{thm:more} hold. 

\vspace{1mm}

\noindent {\rm (i)} The map $(t,x)\mapsto (B^n,\Sigma^n)(t,x,\mub^n_t)$ is H\"older continuous in $t$ and Lipschitz in $x$.

\vspace{1mm}

\noindent {\rm (ii)} For each $t>0$, there exists a sub--probability density $p^n(t,\cdot)$ supported on $(0,\infty)$ such that 
\be*
\mu^n_t(\d x) = \big(1-c^n_0(t)\big)\delta_0(\d x) + p^n(t,x)\d x &\mbox{and}& c^n_0(t) = \int_0^{\infty}p^n(t,x)\d x. 
\ee*
\end{lemma}
\begin{proof}
(i) By assumption, $B^n$ and $\Sigma^n$ have linear growth in $x$. Therefore, there exists $C>0$ such that
\be*
\max_{0\le t\le T} \int_{\R_+}x\mu^n_t(\d x)=\max_{0\le t\le T}\E\big[|X^n_t|\big]\le \E\big[\max_{0\le t\le T}|X^n_t|^q\big]\le C.
\ee*
It is clear that $B^n$ and $\Sigma^n$ inherit the Lipschitz property of $b^n$, $\sigma^n$ and $H^n$ in $x$. Moreover, it follows from Lemma \ref{lemma:existence_c} that there exist $\delta\in (0,1)$ and $C>0$, such that
\be* 
|B^n(t,x,\mub^n_t)-B^n(s,x,\mub^n_s)|+|\Sigma^n(t,x,\mub^n_t)-\Sigma^n(s,x,\mub^n_s)|\le C\big[|t-s|^\delta + \Wc_1(\mub^n_t,\mub^n_s)\big]. 
\ee*
As $\Wc_1(\mub^n_t,\mub^n_s) \le C'\sqrt{|t-s|}$, for some constant $C'>0$, this provides the required result.

\vspace{1mm}

(ii) Define $\tau:=\inf\{t\ge 0: Y^n_t\le 0\}$ and notice that $X^n_t=Y^n_t$ on the event $\{\tau>t\}$, where
\be*
Y^n_t=X^n_0 + \int_0^t B^n(s,Y^n_s,\mub^n_s)\d s + 
\int_0^t \Sigma^n(s,Y^n_s,\mub^n_s) \d W_s,\quad \forall t\ge 0,
\ee*
By Kusuoka \cite{KUSUOKA2017359}, the distribution of $Y_t$ is absolutely continuous with respect to the Lebesgue measure for all $t>0$. As $\P[X^n_t\in A]=\P[Y^n_t\in A,\tau>t] \le \P[Y^n_t \in A]$ for all Borel subset $A\subseteq(0,\infty)$, we deduce that $X^n_t$ inherits the absolute continuity of $Y^n_t$ with respect to the Lebesgue measure. 
\end{proof}

\subsubsection{The case of smooth coefficients}\label{ssec:smooth}

We first establish the desired equality \eqref{eq:c} for the approximating process $X^n$. Notice that $\mub^n_t$ is fixed and we write in Section \ref{ssec:smooth} $B^n(t,x)\equiv B^n(t,x,\mub^n_t)$ and $\Sigma^n(t,x)\equiv \Sigma^n(t,x,\mub^n_t)$ for the sake of simplicity. By construction, there exist $C>0$ and $\gamma:=\delta\wedge 1/2$ such that  
\begin{align*}
& |B^n(t,x)|+ |\Sigma^n(t,x)|+|\partial_x  B^n(t,x)| + |\partial_x\Sigma^n(t,x)| +  |\partial^2_{xx}\Sigma^n(t,x)| \le C, 
\\
 & | B^n(t,x)- B^n(t',x')|+ |\Sigma^n(t,x)-\Sigma^n(t',x')| \le C|t-t'|^{\gamma} + C|x-x'|.
\end{align*}
Notice that $\rho^n(dx)=p_0^n(x)dx$ with a smooth density $p_0^n$. By Garroni and Menaldi \cite[Chapter VI, Lemma 1.10 \& Theorem 2.2]{1992green} and \cite[Theorem 2.2.]{1992green}, the Fokker--Planck equation on the half space
\begin{equation}\label{FPn}
\begin{array}{l}
\partial_t p(t,x) =\frac{1}{2}\partial^2_{xx}\big(\Sigma^n(t,x)^2p(t,x)\big)  - \partial_x\big( B^n(t,x)p(t,x)\big) ,\quad  t, x>0, \\
p(0,x)=p_0^n(x),\quad p(t,0)=0 ,\quad t, x>0,
\end{array}
\end{equation}
has an unique classical solution $p^n$ satisfying 
\beq\label{eq:regularity}
|\partial_t p^n(t,x)|+|\partial_x p^n(t,x)|+|\partial^2_{xx} p^n(t,x)| 
\le 
\frac{C}{t}\int_0^\infty e^{-\frac{(x-y)^2}{Ct}} p_0^n(y)\d y,~~\mbox{for all } t,x>0.
\eeq
Further, it follows from Figalli \cite[Lemma 2.3]{FIGALLI2008109} that the marginal distribution $\mu^n_t=\Lc(X^n_t)$ has the following decomposition
\beq\label{munt}
\mu^n_t(\d x) 
= 
\big(1-c^{n}_0(t)\big) \delta_0(\d x) 
+ 
p^n(t,x)\d x &\mbox{and}& c^{n}_0(t)=\int_{0}^\infty p^n(t,x)\d x,
~~\mbox{for all }t\ge 0.
\eeq
Similarly, the corresponding non--absorbed SDE
\begin{equation}\label{SDERn}
R^n_s = x + \int_t^s B^n(u,R^n_u)\d u + \int_t^s \Sigma^n(u,R^n_u)\d W_u,\quad \forall s\ge t
\end{equation}
has a unique solution $R^{n,t,x}$ with marginal distributions $\Lc(R^{n,t,x}_s)$ absolutely continuous with respect to the Lebesgue measure, i.e. $\Lc(R^{n,t,x}_s)=g^n(t,x,s,y)\d y$, where the density function $g^n$ is the unique solution of the backward Kolmogorov equation parameterized by $s,y>0$:
\begin{equation}\label{FPn}
\begin{array}{l}
\partial_t g^n(t,x,s,y) 
=
- \frac{1}{2}\Sigma^n(t,x)^2\partial^2_{xx}g^n(t,x,s,y)
- B^n(t,x)\partial_xg^n(t,x,s,y), ~\mbox{for all}~ t\in [0,s),~x\in\R, \\
g^n(s,x,s,y)=\delta_{y},~\mbox{for all}~x\in\R.
\end{array}
\end{equation}
\begin{proposition}\label{prop:uniqueness}
The density function $p^n$ introduced in \eqref{munt} satisfies 
\begin{eqnarray}\label{pn}
    p^n(s,y)  = \int_0^\infty p^n_0(x) g^n(0,x,s,y) \d x +  \int_0^s g^n(t,0,s,y)\dot{c}^{n}_0(t)  \d t, \quad \mbox{for all }
    s \ge 0, y>0.
\end{eqnarray}
Moreover, denoting $c^n:=(c^n_0,c^n_1)$ and $\overline{B}^n(t,r):=\1_{\{r>0\}}\big(1,B^n(t,r)^+\big)$, we have 
\begin{eqnarray}\label{eq:cn}
    c^{n}(s)  
    &=&
    \E\left[\overline{B}^{n}\left(s,R^{n, 0,X^n_0}_s\right)\right]
- \int_{c^{n}_0(s)}^1\E\left[\overline{B}^{n}\left(s,R^{n, (c_{0}^{n})^{-1}(u),0}_s\right)\right]\d u,\quad  \forall s\in [0,T].
\end{eqnarray}
where $(c_{0}^{n})^{-1}\!:=\!0\!\vee\!\sup\{t\ge 0:c^n_{0}(t)\!\le\! u\}$ is the left-continuous inverse of the non-increasing function $c^{n}_0$. 
\end{proposition}
\begin{proof}
Denote $A^n:=\frac12(\Sigma^n)^2$. Integrating the Fokker--Planck equation on $(0,\infty)$, it follows from the estimate \eqref{eq:regularity} that 
\be*
\int_0^\infty \!\!\partial_t p^n(t,x)\d x 
= 
\int_0^\infty\!\!\partial^2_{xx}\big(A^n p^n\big)(t,x)\d x 
- \int_0^\infty\!\! \partial_x(B^np^n)(t,x)\d x
=
\big[\partial_{x}\big(A^n p\big)(t,.)\big]_0^\infty - \big[ (B^np^n)(t,.)\big]_0^\infty.
\ee*
Using again \eqref{eq:regularity}, we have for all $t>0$ that $({B^n} p^n)(t,\infty)=0=\partial_{x}\big(A^n p^n\big)(t,\infty)$. As $p^n(t,0)=0$, we see by Fubini's theorem that
\begin{eqnarray}\label{c0-FP}
\dot{c}^{n}_0(t) 
=-\partial_{x}(A^np^n)(t,0),\quad \mbox{for all } t>0.
\end{eqnarray}
Observe next that $\partial_t (p^ng^n) + \partial_x \big(B^np^ng^n\big) - \partial_x \big\{\partial_{x} (A^np^n)g^n - A^np^n\partial_x g^n\big\}= 0.$ Integrating both sides over $(0,s)\times (0,\infty)$, we obtain by appropriately changing the order of integration thanks to Fubini's theorem, and by using the initial and boundary conditions together with \eqref{c0-FP}
\be*
 0 
&\!\!\!\!\!\!=& \!\!\!\!\!
\int_0^s\!\!\int_0^\infty \!\! \partial_t (p^ng^n) \d t \,\d x 
+\! \int_0^s\!\!\int_0^\infty\!\! \partial_x (B^np^ng^n)(t,x)\d t\, \d x 
-\! \int_0^s\!\!\int_0^\infty\!\! \partial_x \big\{\partial_{x} \big(A^np^n\big)g^n\!-\! A^np^n\partial_x g^n\big\}(t,x)\d t \,\d x
\\
&\!\!\!=&\!\!\!
\int_0^\infty p^n(s,x)\delta_y(\d x) - \int_0^\infty \rho(x)g^n(0,x,s,y) \d x 
+  \int_0^s \partial_{x}(A^n p^n)(t,0)g^n(t,0,s,y) \d t
\\
&\!\!\!= &\!\!\!
p^n(s,y) - \int_0^\infty  p^n_0(x)g^n(0,x,s,y) \d x 
-  \int_0^s g^n(t,0,s,y)\dot{c}^{n}_0(t) \d t,  \quad \mbox{for all } s,y>0,
\ee*
which is exactly \eqref{pn}. Integrating $\overline B^n(s,\cdot)p^n(s,\cdot)$ over $(0,\infty)$, one obtains
\be* 
c^n(s) &=& \int_0^\infty \overline B^n(s,y) p^n(s,y)\d y \\
&=& \int_0^\infty\overline B^n(s,y) \int_0^\infty  p^n_0(x)g^n(0,x,s,y) \d x \d y
+  \int_0^\infty\overline B^n(s,y) \int_0^s g^n(t,0,s,y)\dot{c}^{n}_0(t) \d t\d y\\
& =&\E\left[\overline{B}^{n}\left(s,R^{n, 0,X^n_0}_s\right)\right]
+ \int_{0}^s\E\left[\overline{B}^{n}\left(s,R^{n, t,0}_s\right)\right]\d c^n_0(t) \\
&=&  \E\left[\overline{B}^{n}\left(s,R^{n, 0,X^n_0}_s\right)\right]
- \int_{c^{n}_0(s)}^1\E\left[\overline{B}^{n}\left(s,R^{n, (c_{0}^{n})^{-1}(u),0}_s\right)\right]\d u,
\ee*
where the last equality follows from the change of variable $u=c^n_0(t)$.
\end{proof}

\subsubsection{The general case} 

Our next objective is to send $n\to\infty$ in the expression \eqref{pn} of $p^n$. Recall that the sequence $X^n:=Y^n H\big(I^n)$, with $I^n_t:=\inf_{0 \le s \le t} Y^n_s$ and marginals $\mu^n_t=\Lc(X^n_t)$, satisfies the SDE 
\be*
   \d X^n_t = \1_{\{X^n_t>0 \}} \Big[ B^n(t,X^n_t,\mub^n_t)\d t +\Sigma^n(t,X^n_t,\mub^n_t)\d W_t\Big],\mbox{ with } \mub^n_t=\Lc\big(Y^n_t H^n\big(I^n_t \big)\big),
   \mbox{ and } \Lc(X^n_0)=\rho^n.
\ee*
 As we showed in the proof of Theorem \ref{thm:approx_regularity}, after possibly passing to asubsequence, we may assume that $X^n\longrightarrow X$ in $\Wc_1$, where the limit $X$ satisfies the SDE
\be*
   dX_t = \1_{\{X_t>0 \}} \Big[ B(t,X_t,\mu_t)\d t +\Sigma(t,X_t,\mu_t)\d W_t\Big],\quad \mbox{with } \mu_t=\Lc(X_t)  \mbox{ and } \Lc(X_0)=\rho. 
\ee*
In particular, one has for a.e. $(t,x) \in \in[0,T] \x \R$ that
\begin{equation}\label{eq:limBnSigman}
\lim_{n\to\infty} (B^n,\Sigma^n)(t,x,\mub^n_t)
=(B,\Sigma)(t,x,\mu_t),~ \lim_{n\to\infty} c^n_0(t)=c_0(t), ~\mbox{and}~
\lim_{n\to\infty} c^n_1(t,\mub^n_t)=c_1(t,\mu_t).
\end{equation}
\begin{proof}[Proof of Theorem \ref{thm:more}]
By Proposition \ref{prop:uniqueness}, we have for arbitrary $f \in C^\infty_c((0,\infty))$ that
\begin{eqnarray}
    \E \left[ f(X^n_s)\right]
    &=&
    \int_0^\infty \E\left[ f(R^{n,0,x}_s)\right]\;p^n_0(x)\;\d x + \int_0^s \E\left[ f(R^{n,t,0}_s)\right] \dot{c}^{n}_0(t)\d t
    \nonumber\\
    &=&
    \int_0^\infty \E\left[ f(R^{n,0,x}_s)\right]\;\rho^n(\d x) + \int_{c^n_0(s)}^1 \E\left[ f(R^{n,\beta^n_t,0}_s)\right] \d t,
    ~\mbox{for all}~s\ge 0,
\label{eq:fpn}
\end{eqnarray}
where $\beta^n:=(c^n_0)^{-1}$ is the generalized inverse of the map $c^n_0$.
Notice first that $\E \left[ f(X^n_s)\right]\longrightarrow\E \left[ f(X_s)\right]$, by the weak convergence of $(X^n)_{n \ge 1}$. As $\lim_{n \to \infty}\beta^n=(c_0)^{-1}$ a.e. we may pass to the limit in \eqref{eq:fpn} by using \eqref{eq:limBnSigman}, and obtain:  
\begin{align} \label{eq:limit_equation}
    \E \left[ f(X_s)\right]
    =
    \int_0^\infty \E\left[ f(R^{0,x}_s)\right]\;\rho(\d x) + \int_{c_0(s)}^1 \E\left[ f(R^{c_0^{-1}(t),0}_s)\right] \d t,
    ~\mbox{for a.e.}~s \ge 0.
\end{align}
By the arbitrariness of $f$, this proves the existence of a Borel measurable map $p:[0,T] \x [0,\infty) \to [0,\infty)$ satisfying $\E \left[\int_0^T \varphi(t,X_t)\;\d t\right]
    =
    \int_{[0,T] \x \R_+} \varphi(t,x)\;p(t,x)\; \d x\; \d t$, for all $\varphi \in C^\infty_c((0,T) \x (0,\infty))$, where
\begin{align} \label{eq:density_limit}
    p(s,y)  = \int_0^\infty g(0,x,s,y)\rho(\d x) +  \int_{c_0(s)}^1 g((c_0)^{-1}(t),0,s,y)   \;\d t,
~\mbox{for all}~y>0,~\mbox{and a.e.}~s,
\end{align}
with a version $g(t,x,s,y)$ of the density of the process $R^{t,x}_s$ defined by the limiting SDE \eqref{eq:SDE}. Finally the required statement of \Cref{prop:uniqueness} follows by
direct integration of the function $\overline{B}$.
\end{proof}

\section{Uniqueness in the case of constant sign drift coefficient}\label{sec:uniqueness}

This section is devoted to the proof of Theorem \ref{thm:uniqueness}. We specialize the discussion to the case where the coefficients $b,\sigma$ are  independent of the distribution variable and $b$ has constant sign. We aim at justifying the uniqueness of the solution to \eqref{eq:MFG-equilibriumSDE}. 

First, if $b\le 0$, then $(B,\Sigma)=(b,\sigma)$, and the solution to the absorbed SDE 
\be*
   \d X_t = \1_{\{X_t>0 \}} \left[ b(t,X_t)\d t +\sigma(t,X_t)\d W_t\right]
\ee*
is the stopped process $(Y_{t\wedge \tau})_{0\le t\le T}$, where $Y$ is the unique solution to the SDE $ \d Y_t =  b(t,Y_t)\d t +\sigma(t,Y_t)\d W_t. $ Consequently the uniqueness of $X$ is inherited from that of $Y$.

\vspace{1mm}

In the rest of this section, we focus on the non-trivial case $b>0$ on $\R_+$, and recall that the MFG equilibrium SDE in this case is
$$
   \d X_t = \1_{\{X_t>0 \}} \big[B^c(t,X_t)\d t +\Sigma^c(t,X_t)\d W_t\big],
   ~\mbox{where}~
   B^c(t,x):=\frac{b(t,x)+c_1(t)}{1+c_0(t)},
   ~\mbox{and}~ 
   \Sigma^c(t,x):=\frac{\sigma(t,x)}{1+c_0(t)},
$$
where the coefficients $c=(c_0,c_1)$ are defined by
\be* 
c_0(t):=\mu_t[\R_+^*] &\mbox{and} & c_1(t):=\int_{(0,\infty)}b(t,x)\mu_t(\d x),
\ee*
for an arbitrary solution $X$ to \eqref{eq:MFG-equilibriumSDE} with $\mu_t:=\Lc(X_t)$. Throughout this section, we use $R>r>0$ to denote generic constants that may vary from line to line during the proof. For a measurable map $f=(f_0,f_1): \R_+\to\R^2$, set $\|f\|_t:=\|f_0\|_t+\|f_1\|_t$  with $\|f_i\|_t:=\sup_{0\le u\le t}|f_i(u)|$. By a straightforward computation, there exist $R>r>0$  such that for all maps $f,f':\R_+\to\R^2_+$ 
\beq 
&\|B^f\|\vee\|\Sigma^f\|\vee\|\partial_xB^f\|\vee\|\partial_x\Sigma^f\|
\le 
R , ~~\inf_{(t,x)}\Sigma^f(t,x)\ge r,&
\label{bound-coefs}
\\ 
&\|(B^f-B^{f'})(t,\cdot)\| 
+ \|(\Sigma^{f}-\Sigma^{f'})(t,\cdot)\|
 + \|(\partial_x\Sigma^{f}-\partial_x\Sigma^{f'})(t,\cdot)\|
 \le 
 R\,\|f-f'\|_t,~ \mbox{for all}~ t\in [0,T].~~~&
 \label{diff-coefs} 
\eeq  
\begin{lemma} \label{lemm:continuity}
Let the conditions of Theorem \ref{thm:uniqueness} hold. 

\vspace{1mm}

\noindent {\rm (i)} The function $c:[0,T]\to [0,1] \x \R_+$ is H\"older continuous. In particular,  $B^c,\Sigma^c$ are H\"older continuous in $t$ and Lipschitz in $x$.

\vspace{1mm}

\noindent {\rm (ii)} For each $t>0$, there exists a sub--probability density $p(t,\cdot)$ supported on $(0,\infty)$ such that 
\be*
\mu_t(\d x) = \big(1-c_0(t)\big)\delta_0(\d x) + p(t,x)\d x ,&\mbox{and}& c_0(t) = \int_0^{\infty}p(t,x)\d x. 
\ee*
\end{lemma}
\begin{proof}
The proof of (ii) is as same as that for Lemma \ref{lem:decom} and we thus omit it. We start proving (i) for $c_0$. Let $Y$ be the unique solution to the SDE 
\be*
   \d Y_t = B^c(t,Y_t)\d t +\Sigma^c(t,Y_t)\d W_t,\quad \forall t\ge 0
\ee*
such that $Y_0=X_0$. Then it follows that $X_t=Y_{t\wedge \tau}$ with $\tau:=\inf\{t\ge 0: Y_t\le 0\}$. Hence, $c_0(t)=\P[\tau>t]$ is H\"oder continuous by  \Cref{lem:ito-martingale} (i). The H\"older continuity of $c_1$ follows from  the probabilistic representation
\be* 
c_1(t) = \E[b(t,X_t){\1}_{\{X_t>0\}}] = \E[b(t,X_t)] - \E[b(t,X_t){\1}_{\{X_t=0\}}] = \E[b(t,X_t)] - b(t,0)(1-c_0(t)).
\ee*
The required H\"older continuity now follows by standard estimates. Finally the H\"older continuity of $B^c$ and $\Sigma^c$ in $t$ follows by immediate composition, while the Lipschitz continuity in $x$ is directly inherited from that of $b$ and $\sigma$.
\end{proof}
Thanks to the H\"older continuity of the coefficients $B^c,\Sigma^c$ in $t$, the autonomous characterization \eqref{eq:c} of $c$ in \Cref{prop:uniqueness} also holds. Hence, we see that the uniqueness of $X$ can be derived by showing the uniqueness of such $c$. We first rewrite the representation \eqref{eq:c} by using the notation $\vec{b}_0(s,y):=\big(1, b(s,y)\big)$, and after an integration by parts:
\be* 
c(s)  
&=& 
\int_0^\infty  p_0(x)\d x\int_0^\infty  \vec{b}_0(s,y)g^c(0,x,s,y) \d y +  \int_0^s \dot{c}_0(t) \d t \int_0^\infty \vec{b}_0(s,y)g^c(t,0,s,y) \d y
\\
&=& 
\int_0^\infty \int_0^\infty \vec{b}_0(s,y)g^c(0,x,s,y)\rho(x)   \d x\,\d y 
+\frac{c_0(s)}{2}\vec{b}_0(s,0)
 \\
 &&
 -\int_0^\infty  \vec{b}_0(s,y)g^c(0,0,s,y) \d y 
 -\int_0^s \int_0^\infty  \vec{b}_0(s,y)\partial_t g^c(t,0,s,y)c_0(t) \d y\, \d t  
\ee*
where we recall that $g^c(t,x,s,\cdot)$ denotes the density function of $R^{t,x}_s$.

Let $\Cc_+([0,T])\subset C([0,T])$ be the subset of non-negative H\"older continuous functions. Define the operator $\Lambda=(\Lambda_0,\Lambda_1):\Cc_+([0,T])^2 \to C([0,T])^2$, defined for all $f\in \Cc_+([0,T])^2$ and $s\in [0,T]$ by
\be*  
\Lambda[f](s)
\!:=\! 
\!\!\int_0^\infty\!\!\!\!\! \int_0^\infty \!\!\!\!p_0(x)\vec{b}(s,y)g^f\!(0,x,s,y) \d x\,\d y 
- \!\!\int_0^\infty\!\! \vec{b}(s,y)g^f\!(0,0,s,y) \d y
- \!\!\int_0^s\!\!\! \int_0^\infty \!\!\!\!f_0(t) \vec{b}(s,y)\partial_t g^f\!(t,0,s,y) \d y\, \d t,
\ee* 
with obvious definition of $g^f(t,x,s,\cdot)$ and $\vec{b}:=(2,b)$. Define further $\overline{\Lambda}$ by 
\be* 
\overline{\Lambda}[f]
:=
\big(\Lambda_0[f], \Lambda_1[f]+b_0f_0\big),\quad \mbox{with } b_0:=\frac12 b(\cdot,0).
\ee* 
Therefore, any $c$ corresponding to a solution $X$ of \eqref{eq:MFG-equilibriumSDE} must be a fixed point of $\overline{\Lambda}$, i.e. $c=\overline{\Lambda}[c]$, and the following uniqueness result of $c$ implies the required uniqueness result of Theorem \ref{thm:uniqueness}.
 
 \begin{proposition}\label{prop:ac}
$\overline{\Lambda}$ has at most one fixed point on $\Cc_+([0,T])^2$ and thus the solution to \eqref{eq:MFG-equilibriumSDE} is unique.
\end{proposition}
To prove Proposition \ref{prop:ac}, we need some stability  estimates for the map $f\longmapsto g^f$, which are summarized in Proposition \ref{lem1}. Recall the Parametrix expressions of the density $g^f$ for $f\in \Cc_+([0,T])^2$, see e.g. Aronson \cite{AronsonBounds} and Konakov, Kozhina \& Menozzi \cite{Konakov2015STABILITYOD}:
\be*
g^f(t,x,s,y) 
= 
\sum_{k=0}^{\infty} g^f_0\!\otimes\! G_f^{(k)}(t,x,s,y), 
~\mbox{with}~  g^f_0\!\otimes\! G_f^{(0)}:=g^f_0
~\mbox{and}~ g^f_0\!\otimes\! G_f^{(k)}:=\big(g^f_0\!\otimes\! G_f^{(k-1)}\big)\!\otimes\! G_f,~k\ge 1,
\ee*
where we used the space-time convolution notation
\be*
\psi_1\otimes\psi_2(t,x,s,y) := \int_t^s\!\!\!\!\int_{-\infty}^\infty \psi_1(t,x,u,z)\psi_2(u,z,t,y)\d z\,\d u,
\ee* 
for all scalar functions $\psi_1,\psi_2$ defined on the appropriate spaces, and
\begin{eqnarray*} 
g^f_0(t,x,s,y)
&:=&
\phi\big(A^f(t,s,y),y-x\big),
\quad \mbox{with }
A^f(t,s,y):=\int_t^s \Sigma^f(u,y)^2\d u, ~~
\phi(u,z):=\frac{1}{\sqrt{2\pi u}}e^{\frac{-z^2}{2u}},
\\
G_f(t,x,s,y) 
&:=& \frac12\big(\Sigma^f(t,x)^2-\Sigma^f(t,y)^2\big)\partial^2_{xx}g_0^f(t,x,s,y) 
                           + B^f(t,x)\partial_{x}g_0^f(t,x,s,y).
\end{eqnarray*}
\begin{proof}[Proof of Proposition \ref{prop:ac}]
We argue by contradiction. 
Let $c,c' \in \Cc_+([0,T])^2$ be two different fixed points of $\overline\Lambda$ and suppose to the contrary that $t^*:=\inf\{t>0: c(t)\neq c'(t)\}<T$. We claim that 
\beq \label{eq:contraction}
\|\Lambda[c] - \Lambda[c']\|_s \le 
R\sqrt{(s-t^*)^+}\,\|c-c'\|_s,\quad \forall s\in [0,T].
\eeq
for some $R>0$. Before proving this, let us show that it induces the required contradiction. For $s>t^*$
\be* 
\|c-c'\|_s
=
\big\|\overline\Lambda[c] - \overline\Lambda[c']\big\|_s 
&\!\!\le&\!\!
\|\Lambda[c] \!-\! \Lambda[c']\|_s
+\frac12 \|b_0\|_T \|c_0 \!-\! c'_0\|_s
\\
&\!\!=&\!\!
\|\Lambda[c] \!-\! \Lambda[c']\|_s 
+\frac12 \|b_0\|_T \|\Lambda_0[c] \!-\! \Lambda_0[c']\|_s
\;\le\;
\Big(1+\frac12 \|b_0\|_T\Big)\|\Lambda[c] \!-\! \Lambda[c']\|_s, 
\ee* 
as $\Lambda$ and $\overline\Lambda$ have the same first component $\Lambda_0=\overline\Lambda_0$. Then, it follows from \eqref{eq:contraction} that 
$$0\;<\; \|c-c'\|_s\le 
R\Big(1+\frac{1}{2}\|b_0\|_T\Big)\sqrt{(s-t^*)^+}\|c-c'\|_s,$$
which yields a contradiction as $s>t^*$ is sufficiently close to $t^*$. 

It remains to prove  \eqref{eq:contraction}. Without loss of generality, we may reduce to the case $t^*=0$. Recall that 
\be*  
\Lambda[f](s)
\!:=\! 
\!\!\int_0^\infty\!\!\!\!\! \int_0^\infty \!\!\!\!p_0(x)\vec{b}(s,y)g^f\!(0,x,s,y) \d x\,\d y 
- \!\!\int_0^\infty\!\! \vec{b}(s,y)g^f\!(0,0,s,y) \d y
- \!\!\int_0^s\!\!\! \int_0^\infty \!\!\!\!f_0(t) \vec{b}(s,y)\partial_t g^f\!(t,0,s,y) \d y\, \d t,
\ee* 
for $f\in \Cc_+([0,T])^2$, which yields 
\beq  \label{ineq:decomp}
\big| \Lambda[c](s)-\Lambda[c']\big|(s)
&\le& 
\Big|\int_0^\infty  p_0(x)I_s(x)\d x\Big| + |I_s(0)| +\|c-c'\|_s\int_0^s J_s(t) \d t + \int_0^s K_s(t) \d t,
\\
\mbox{with}
&&
\hspace{-9mm}
I_s(x) \!:=\!\int_0^\infty\!\!\vec{b}(s,y) (g^c\!-\!g^{c'})(0,x,s,y)\d y,~
J_s(t) \!:=\!\Big| \int_0^{\infty}\!\! \vec{b}(s,y) \partial_t g^c(t,0,s,y)\d y\Big|,
\nonumber\\
\mbox{and}
&&
\hspace{-9mm}
K_s(t) :=\Big| \int_0^{\infty} \!\!\vec{b}(s,y)\partial_t (g^c\!-\! g^{c'})(t,0,s,y)\d y \Big|. 
\nonumber
\eeq 
Applying (componentwise) the estimates \eqref{ineq:stab1} and \eqref{ineq:stab2} of \Cref{lem1} below, one has for $t=0$:
$$ 
\left|\int_0^\infty p_0(x) I(x)\d x\right| 
 \le 
 \sum_{k=0}^\infty
 \int_0^\infty p_0(x) \d x \left |\int_0^\infty\vec{b}(s,y)\big(g^c_0\otimes G_c^{(k)}-g^{c'}_0\otimes G_{c'}^{(k)}\big)
                                                                                  (0,x,s,y)\d y \right | 
 \le 
 R\sqrt{s}\|c-c'\|_s,
$$
 as $\sup_{\lambda>0} \int_{0}^\infty p_0(\lambda x)e^{-x^2/2}\d x<\infty$ by Condition (iii) of Theorem \ref{thm:uniqueness}, and
 \be*
 \left|I(0)\right| 
 \le  
 \sum_{k=0}^\infty 
 \left |\int_0^\infty\vec{b}(s,y)\big(g^c_0\otimes G_c^{(k)}-g^{c'}_0\otimes G_{c'}^{(k)}\big)(0,0,s,y)\d y \right |
 \le R\sqrt{s}\|c-c'\|_s.
\ee* 
Applying (again componentwise) the estimates \eqref{ineq:stab5} and \eqref{ineq:stab6} of \Cref{lem1} below, it follows that
\be* 
\int_0^s J_s(t) \d t 
\le  
\sum_{k=0}^\infty\int_0^t \Big|\int_0^\infty \vec{b}(s,y)\partial_t g^f_0\otimes G_f^{(k)}(t,0,s,y)\d y\Big|\d t 
\le 
R\int_0^s\sqrt{s-t}\;\d t \le R\sqrt{s}, 
\ee*
Finally, by \eqref{ineq:stab3} and \eqref{ineq:stab4} of \Cref{lem1} below, we get:
\be* 
\int_0^s K_s(t) \d t 
\le  
\sum_{k=0}^\infty\int_0^t \Big|\int_0^\infty \vec{b}(s,y)\partial_t\big(g^f_0\otimes G_f^{(k)}-g^{f'}_0\otimes G_{f'}^{(k)}\big)(t,0,s,y)\d y\Big|\d t \le R\sqrt{s}\|c-c'\|_s. 
\ee*
Plugging the last estimates in \eqref{ineq:decomp} yields \eqref{eq:contraction} for $t^*=0$, and thus completing the proof. 
\end{proof}
The last proof refers to the following statement which uses the Gamma function $\Gamma(z):=\int_0^\infty a^{z-1}e^{-a}\d a$.

\begin{proposition}\label{lem1}
Let $f,f'\in \Cc_+([0,T])^2$, and $\beta: \R_+\times\R\to\R$ be with bounded derivative. Then, under the conditions of Theorem \ref{thm:uniqueness}, there exist $R>r>0$ s.t. the following inequalities hold for all $0\le t<s\le T$, $x\ge 0$, and $k\ge 1$:
\beq 
\Big|\int_0^\infty \!\!\! \beta(s,y)\big(g_0^f- g_0^{f'}\big)(t,x,s,y)\d y\Big| 
\hspace{5mm}&\hspace{-14mm}\le&\!\!\!\!\!\!\!\!\!\!\!\! \!\!\!\!
\big[x \phi(r(s-t)^{\frac12},x)+(1+x)(s-t)^{\frac12}\big]R \|f-f'\|_s, ~~~
\label{ineq:stab1}\\
\Big|  \int_0^\infty \!\!\!\beta(s,y)\big(g^f_0\otimes G_f^{(k)}-g^{f'}_0\otimes G_{f'}^{(k)}\big)(t,x,s,y) \d y \Big|  
&\!\!\le &\!\!  \frac{R^{k+1}(1+x)(s-t)^{\frac{k}{2}}}{\Gamma(1+\frac{k}{2})}\|f-f'\|_s,
\label{ineq:stab2}\\
\Big| \int_0^{\infty} \!\!\!\beta(s,y)\partial_t(g_0^f- g_0^{f'})(t,0,s,y)\d y \Big| 
&\!\!\le& \!\!
R(s-t)^{-\frac12}\|f-f'\|_s, 
\label{ineq:stab3}\\
\Big|\int_0^\infty \!\!\!\beta(s,y)\partial_t\big(g^f_0\otimes G_f^{(k)}-g^{f'}_0\otimes G_{f'}^{(k)}\big)(t,0,s,y)\d y\Big| 
&\!\!\le&  \!\!
\frac{R^{k+2}}{\Gamma(\frac{k}{2}-1/2)}(s-t)^{\frac{k}{2}-1}\|f-f'\|_s,
\label{ineq:stab4}\\
\Big| \int_0^{\infty} \!\!\!\beta(s,y)\partial_t g_0^f(t,0,s,y)\d y \Big| 
&\!\!\le& \!\!
R(s-t)^{-\frac12}, 
\label{ineq:stab5}\\
\Big|\int_0^\infty \!\!\!\beta(s,y)\partial_t g^f_0\otimes G_f^{(k)}(t,0,s,y)\d y\Big| 
&\!\!\le &  \!\!
\frac{R^{k+2}}{\Gamma(\frac{k}{2}-\frac12)}(s-t)^{\frac{k}{2}-1}.
\label{ineq:stab6}
\eeq
\end{proposition}
\begin{proof}
Denote  $\Delta:=A^{f'}-A^f$ and $\overline{A}_{\lambda}:=A^f+\lambda \Delta$. Throughout this proof, $R>r>0$ are constants which may change from line to line.

\vspace{1mm}

(i) We first prove estimates \eqref{ineq:stab1}, \eqref{ineq:stab3}, and \eqref{ineq:stab5} corresponding to $k=0$. We start by writing that
\begin{eqnarray}\label{estim1}
\int_0^\infty \beta(s,y)(g^f_0-g^{f'}_0)(t,x,s,y)\d y
=
\int_0^1\int_0^\infty \beta(y) \Delta(t,s,y)\partial_r\phi\big(\overline{A}_{\lambda}(t,s,y),y-x\big)\d y\,\d\lambda.
\end{eqnarray}
Dropping the arguments $t$ and $s$ which play no role here, we compute that
\be*
\partial_r\phi(\overline{A}_{\lambda}(y),y-x)
&\hspace{-3mm}=&\hspace{-3mm}
\Big(\frac{-1}{2\overline{A}_{\lambda}(y)}
       +\frac{(y-x)^2}{2\overline{A}_{\lambda}(y)^2}
\Big)
\frac{1}{\overline{A}_{\lambda}(y)^{1/2}}
\phi\Big(1,\frac{y-x}{\overline{A}_{\lambda}(y)^{1/2}}\Big)
\\
&\hspace{-3mm}=&\hspace{-3mm}
\frac{-(y-x)}{2\overline{A}_{\lambda}(y)^{3/2}}
\frac{\partial}{\partial y}\Big\{\phi\Big(1,\frac{y-x}{\overline{A}_{\lambda}(y)^{1/2}}\Big)\Big\}
+\Big(\frac{(y-x)^3\partial_y\overline{A}_{\lambda}(y)}{4\overline{A}_{\lambda}(y)^{7/2}}
         - \frac{1}{2\overline{A}_{\lambda}(y)^{3/2}}
\Big)\phi\Big(1,\frac{y-x}{\overline{A}_{\lambda}(y)^{1/2}}\Big).
\ee*
Then, it follows from  direct integration by parts that
\beq 
\int_0^\infty\hspace{-2mm}
 (\beta\Delta)(y)\partial_r\phi\big(\overline{A}_{\lambda}(y),y\big)\d y 
&\hspace{-3mm}=&\hspace{-3mm}
\frac{-x(\beta\Delta)(0)}{2\overline{A}_{\lambda}(0)^{3/2}}\phi\Big(1,\frac{x}{\overline{A}_{\lambda}(0)^{1/2}}\Big)
\!+\!\int_0^\infty\hspace{-2mm} 
         \frac{\partial}{\partial y}\Big\{\frac{(y-x)(\beta\Delta)(y)}
                                                       {2\overline{A}_{\lambda}(y)^{3/2}}
                                       \Big\}
                  \phi\Big(1,\frac{y-x}{\overline{A}_{\lambda}(y)^{1/2}}\Big)\d y\nonumber
\\
&\hspace{-3mm}&\hspace{-3mm}
+\int_0^\infty \beta(y) \Delta(y)
                     \Big(\frac{(y-x)^3\partial_y\overline{A}_{\lambda}(y)}{4\overline{A}_{\lambda}(y)^{7/2}}
                             - \frac{1}{2\overline{A}_{\lambda}(y)^{3/2}}
                     \Big)
                     \phi\Big(1,\frac{y-x}{\overline{A}_{\lambda}(y)^{1/2}}\Big)
                     \d y
 \nonumber\\
&\hspace{-3mm}=&\hspace{-3mm}
\frac{-x\beta(0)\Delta(0)}{2\overline{A}_{\lambda}(0)}\phi\big(\overline{A}_{\lambda}(0),x\big)
 \label{eq:expression}  \\ 
&\hspace{-3mm}&\hspace{-15mm}
+\int_0^\infty \Big[(y-x)\frac{\partial}{\partial y}
                                    \Big\{\frac{\beta(y) \Delta(y)}
                                           {2\overline{A}_{\lambda}(y)^{3/2}}
                                     \Big\}
                                                  +\beta(y) \Delta(y)
                                                    \frac{(y-x)^3\partial_y\overline{A}_{\lambda}(y)}{4\overline{A}_{\lambda}(y)^{7/2}}
                                            \Big]
                                            \phi\Big(1,\frac{y-x}{\overline{A}_{\lambda}(y)^{1/2}}\Big)
                     \d y.\nonumber 
\eeq

Note that, under our conditions, $r(s-t) \le \overline{A}_{\lambda} \le R(s-t)$, $|\dot{\beta}|\le R$, $|\partial_y\overline{A}_{\lambda} |\le R(s-t)$, 
$|\Delta|\le R(s-t)\|f-f'\|_s$ and $|\partial_y \Delta|\le R(s-t)\|f-f'\|_s$ hold for some $R>r>0$. Plugging these estimates in the last expression, and substituting in \eqref{estim1}, directly yields \eqref{ineq:stab1}. For later use, notice that the same argument also provides 
\beq 
\Big|\int_{-\infty}^\infty  \beta(s,y)\big(g_0^f- g_0^{f'}\big)(t,x,s,y)\d y\Big| 
&\le& 
R(s-t)^{1/2} \|f-f'\|_s.
\label{ineq:stab11}
\eeq

\vspace{1mm}
\quad Fix $x=0$. Differentiating both sides of \eqref{eq:expression} with respect to $t$ and integrating wrt to $y$, we obtain 
\be*
&\hspace{-3mm}&\hspace{-3mm}
\int_0^\infty \beta(s,y)(\partial_tg^f_0-\partial_tg^{f'}_0)(t,0,s,y)\d y \\
&\hspace{-3mm}=&\hspace{-3mm}
\frac{\partial}{\partial t} \int_0^1
\int_0^\infty \hspace{-1mm}
                  \Big[y\frac{\partial}{\partial y}\Big\{\frac{\beta(s,y) \Delta(t,s, y)}
                                                                                 {2\overline{A}_{\lambda}(t,s,y)^{3/2}}
                                                                      \Big\}
                                                  +\beta(s,y) \Delta(t,s,y)
                                                    \frac{y^3\partial_y\overline{A}_{\lambda}(t,s,y)}{4\overline{A}_{\lambda} (t,s,y)^{7/2}}
                                            \Big]
                                            \phi\Big(1,\frac{y}{\overline{A}_{\lambda}(t,s,y)^{1/2}}\Big)
                     \d y \d\lambda, \\
                     &\hspace{-3mm}=&\hspace{-3mm}
 \int_0^1
\int_0^\infty\hspace{-1mm}
                  \Big[y\frac{\partial^2}{\partial y\partial t}\Big\{\frac{\beta(s,y) \Delta(t,s, y)}
                                                                                 {2\overline{A}_{\lambda}(t,s,y)^{3/2}}
                                                                      \Big\}
                                                  +\beta(s,y) \frac{\partial}{\partial t} \Big\{\Delta(t,s,y)
                                                    \frac{y^3\partial_y\overline{A}_{\lambda}(t,s,y)}{4\overline{A}_{\lambda} (t,s,y)^{7/2}}\Big\}
                                            \Big]
                                            \phi\Big(1,\frac{y}{\overline{A}_{\lambda}(t,s,y)^{1/2}}\Big)
                     \d y \d\lambda \\
                     &\hspace{-3mm}=&\hspace{-3mm}
\int_0^1
\int_0^\infty \hspace{-1mm}
                   \Big[y\frac{\partial}{\partial y}\Big\{\frac{\beta(s,y) \Delta(t, s,y)}
                                                                                 {2\overline{A}_{\lambda}(t,s,y)^{3/2}}
                                                                      \Big\}
                                                  +\beta(s,y) \Delta(t,s,y)
                                                    \frac{y^3\partial_y\overline{A}_{\lambda}(t,s,y)}{4\overline{A}_{\lambda} (t,s,y)^{7/2}}
                                            \Big]
                                            \frac{\partial}{\partial t} \phi\Big(1,\frac{y}{\overline{A}_{\lambda}(t,s,y)^{1/2}}\Big)
                     \d y \d\lambda, \\
\ee*
Using the facts $\| \partial_t\overline{A}_{\lambda}\|+\| \partial^2_{ty}\overline{A}_{\lambda}\|\le R$
and $\|\partial_t\Delta\|+\| \partial^2_{ty}\Delta\|\le R \|f-f'\|_s$, we obtain the estimate \eqref{ineq:stab3} by the same reasoning of (i). For later use, notice that the same argument also provides 
\beq 
\Big| \int_{-\infty}^{\infty} \beta(s,y)\partial_t(g_0^f- g_0^{f'})(t,0,s,y)\d y \Big| 
&\le& 
R(s-t)^{-1/2}\|f-f'\|_s.
\label{ineq:stab33}
\eeq
By following similar arguments as above, we also obtain the estimate \eqref{ineq:stab5}, together with the following estimate which wil be needed later:
\beq\label{ineq:stab55}
\Big| \int_{-\infty}^{\infty} \beta(s,y)\partial_t g_0^f(t,0,s,y)\d y \Big| 
&\le& 
R(s-t)^{-1/2}.
\eeq
(ii) We next focus on \eqref{ineq:stab2}. By direct adaptation from Konakov, Kozhina \& Menozzi \cite{Konakov2015STABILITYOD}, there exist constants $R, r>0$ such that, for $0\le t<s \le T$, $x,y\in\R$ and $k\ge 1$,
\beq  
\big| G_f^{(k)}(t,x,s,y)\big |
&\le&  
\frac{R^k}{\Gamma(\frac{k+1}{2})}\phi\big(r(s-t),y-x\big)(s-t)^{k/2-1}
\label{ineq:esti1}
\\
\big| G_f^{(k)}(t,x,s,y)- G_{f'}^{(k)}(t,x,s,y)\big | 
&\le&  
\frac{R^k}{\Gamma(\frac{k+1}{2})}\phi\big(r(s-t),y-x\big)(s-t)^{k/2-1}\|f-f'\|_s,
\label{ineq:esti2} 
\eeq 
Combining with \eqref{ineq:stab11}, we deduce that for all $k\ge 1$: 
\beq
\left|g^f_0\otimes G_f^{(k)}-g^{f'}_0\otimes G_{f'}^{(k)}\right|(t,x,s,y)  \le   \frac{R^{k+1}}{\Gamma(\frac{k+1}{2})}\phi(r(s-t),y-x)(s-t)^{k/2}\|f-f'\|_s,
\label{ineq:esti3}
\eeq
which yields \eqref{ineq:stab2} by direct integration with respect to $y$ over $(0,\infty)$, after multiplying by $\beta(s,y)$. 

\vspace{3mm}

(iii) The rest of this proof justifies \eqref{ineq:stab4}; the remaining estimate \eqref{ineq:stab6} is also proved by similar arguments that we do not report here for brevity. By definition, we have $\partial_t\big(g_0^f\otimes G_f^{(k)}\big)(t,0,s,y)=O_f^{(k)}(t,s,y)-G_f^{(k)}(t,0,s,y)$,  where
\begin{align*} 
O_f^{(k)}(t,s,y)
:=
\int_t^s\!\!\!\int_{-\infty}^\infty\partial_t g_0^f(t,0,u,z) G_f^{(k)}(u,z,s,y)\d z \d u,
~~k\ge 1.
\end{align*}
Then, 
\begin{align} 
\Big|\int_0^\infty \!\!\!\beta(s,y)\partial_t\big(g^f_0\otimes G_f^{(k)}
                &-g^{f'}_0\otimes G_{f'}^{(k)}\big)(t,0,s,y)\d y\Big| 
\nonumber\\
&\le
\Big|\int_0^\infty\!\!\! \beta(s,y)\big(\big(O_f^{(k)}-O_{f'}^{(k)}\big)(t,s,y)-\big(G_{f}^{(k)}- G_{f'}^{(k)}\big)(t,0,s,y) 
                                               \big)\d y \Big|
\nonumber\\
&
\le
 \frac{R^{k}}{\Gamma(\frac{1+k}{2})}(s-t)^{\frac{k}{2}-1}\|f-f'\|_s
 +\Big|\int_0^\infty\!\!\!\beta(s,y)\big(O_f^{(k)}-O_{f'}^{(k)}\big)(t,s,y)\d y\Big|,
\label{estim:last}
\end{align}
by \eqref{ineq:esti2} together with the fact that $\int_0^\infty|\beta(y)| \phi(r(s-t),y) \d y<\infty$ as $\beta$ has affine growth. We next decompose
\be*
\big(O_f^{(k)}-O_{f'}^{(k)}\big)(t,s,y)
&=&
\int_t^s \big(\Delta^{(k)}_1 + \Delta^{(k)}_2 + \Delta^{(k)}_3\big)(u)du,
\ee*
where
\begin{itemize}
\item $\Delta^{(k)}_1(u):=G_{f'}^{(k)}(u,0,s,y)\int_{-\infty}^\infty \partial_t\big(g_0^f-g_0^{f'}\big)(t,0,u,z) \d z$, satisfies by \eqref{ineq:stab33} and \eqref{ineq:esti1} 
\begin{align*}
|\Delta^{(k)}_1(u)|
&\le \frac{R^k}{\Gamma(\frac{1+k}{2})}(s-u)^{\frac{k}{2}-1}\phi\big(r(s-u),y\big)R\|f-f'\|_s (u-t)^{-\frac{1}{2}}
\\
&
\le \frac{R^{k+1}}{\Gamma(\frac{1+k}{2})}\|f-f'\|_s (s-t)^{\frac{k-3}{2}}\phi\big(r(s-u),y\big)(u-t)^{-\frac{1}{2}},
\end{align*}
and therefore
\begin{eqnarray}\label{estim:last1}
\Big|\int_0^\infty\!\!\!\beta(s,y) \int_t^s\Delta^{(k)}_1(u)\d u\;\d y\Big|
&\le&
\frac{R^{k+2}}{\Gamma(\frac{1+k}{2})}\|f-f'\|_s (s-t)^{\frac{k-3}{2}},
\end{eqnarray}
as $\int_0^\infty|\beta(y)| \phi(r(s-t),y) \d y<\infty$ due to the affine growth of $\beta$ in $y$;

\item $\Delta^{(k)}_2(u):=\int_{-\infty}^\infty \partial_t g_0^{f'}(t,0,u,z) \big(G_{f'}^{(k)}-G_f^{(k)}\big)(u,0,s,y)\d z$ satisfies by \eqref{ineq:esti2} and \eqref{ineq:stab55} 
$$
|\Delta^{(k)}_2(u)| \le R(u-t)^{-\frac{1}{2}}\frac{R^k}{\Gamma(\frac{k+1}{2})}\phi\big(r(s-u),y-z\big)(s-u)^{\frac{k}{2}-1}\|f-f'\|_s,
$$
so that as before:
\begin{eqnarray}\label{estim:last2}
\Big|\int_0^\infty\!\!\!\beta(s,y) \int_t^s\Delta^{(k)}_2(u)\d u\;\d y\Big|
&\le&
\frac{R^{k+2}}{\Gamma(\frac{k+1}{2})}(s-t)^{\frac{k-3}{2}}\|f-f'\|_s.
\end{eqnarray}
\item and $\Delta^{(k)}_3(u):=\int_{-\infty}^\infty \partial_t\big(g_0^f-g_0^{f'}\big)(t,0,u,z) \big( G_{f'}^{(k)}(u,z,s,y)-G_{f'}^{(k)}(u,0,s,y)\big)\d z$, and is now the only remaining term to complete the proof of the required estimate \eqref{ineq:stab4}. 
To do so, we need the following estimation and we start with $k=1$. As $G_{f'}(.,x,.)-G_{f'}(.,0,.)=\int_0^x\partial_x G_{f'}(.z,.)dz$, we have
\be* 
&&
|G_{f'}(u,x,s,y)-G_{f'}(u,0,s,y)| 
\\
&=&
\Big|\int_0^x\left[  \Big(\frac{\sigma(t,z)\partial_x\sigma(t,z)}{\big(1+f'_0(t)\big)^2} + B'(t,x)\Big|)\partial^2_{xx} \tq(t,z,s,y) + \frac{\sigma(t,z)^2-\sigma(t,y)^2}{2\big(1+f'_0(t)\big)^2}\partial^3_{xxx}\tq(t,z,s,y)\right] \d z\Big| 
\\
&\le &
R\int_0^x\phi(r(s-t),y-z)) \left(\frac{1}{s-t} + \frac{|y-z|}{(s-t)^{3/2}}\right) \d z.
\ee* 
Hence,
\be* 
\int_0^\infty\hspace{-2mm} 
                  |G_{f'}(u,x,s,y)\!-\!G_{f'}(u,0,s,y)|\d y 
\le 
R\!\int_0^{\infty}\hspace{-2.5mm}\int_0^x\hspace{-2mm}
       \phi(r(s-t),y-z)) \!\Big(\frac{1}{s-t} + \frac{|y-z|}{(s-t)^{3/2}}\Big) \d z \d y 
\le 
\frac{R|x|}{s-t}.
\ee* 
Fix any $\gamma\in (0,1)$, e.g. $\gamma=1/2$. We distinguish two cases.  If $|x|\le \sqrt{s-t}$, then
\be* 
\int_0^\infty |G_{f'}(u,x,s,y)-G_{f'}(u,0,s,y)|\d y 
\le \frac{R|x|}{s-t}\le \frac{R|x|^{\frac 1 2}(s-t)^{\frac 1 4}}{s-t}=  \frac{R|x|^{\gamma}}{(s-t)^{\frac 3 4}}.
\ee* 
If $|x|>\sqrt{s-t}$, then
\be* 
|G_{f'}(u,x,s,y)-G_{f'}(u,0,s,y)| &\le & \big[|G_{f'}(u,x,s,y)|+|G_{f'}(u,0,s,y)|\big] \frac{|x|^\frac 1 2}{(s-t)^{\frac 1 4}} \\
&\le&  R\frac{\phi(r(s-t),y-x)+\phi(r(s-t),y)}{\sqrt{s-t}}\frac{|x|^\frac 1 2}{(s-t)^{\frac 1 4}}
\ee*
and thus
\be* 
\int_0^\infty |G_{f'}(u,x,s,y)-G_{f'}(u,0,s,y)|\d y 
\le \frac{R|x|^{\frac 1 2}}{(s-t)^{\frac 3 4}}.
\ee* 
We compute for all $k\ge 1$
\be* 
&\hspace{-2mm}&\int_0^\infty |G^{(k)}_{f'}(u,x,s,y)-G^{(k)}_{f'}(u,0,s,y)|\d y
\\
&\hspace{-2mm}=&\hspace{-3mm}
 \int_0^\infty \left|\int_t^s \int_{-\infty}^\infty \big(G_{f'}(t,x,u,z)-G_{f'}(t,0,u,z)\big)G^{(k-1)}_{f'}(u,z,s,y)\d z \d u\right|\d y 
 \\
&\hspace{-2mm}\le &\hspace{-3mm}
\int_t^s \d u \int_{-\infty}^\infty \big |G_{f'}(t,x,u,z)-G_{f'}(t,0,u,z)\big| \d z \int_0^\infty \big|G_{f'}^{(k-1)}(u,z,s,y)\big| \d y 
\\
&\hspace{-2mm}\le &\hspace{-3mm}
\int_t^s \d u \int_{-\infty}^\infty \big |G_{f'}(t,x,u,z)-G_{f'}(t,0,u,z)\big| \d z \int_0^\infty \frac{R^{k-1}\Gamma(\frac 1 2)^{k-2}}{\Gamma(\frac{k-1}{2})}\phi(r(s-u),y-z)(s-u)^{\frac{k-3}{2}} \d y 
\\
&\hspace{-2mm}\le &\hspace{-3mm}
\int_t^s  \frac{|x|^{\frac 1 2}}{(u-t)^{\frac  3 4}}  \frac{R^{k}\Gamma(\frac 1 2)^{k-2}}{\Gamma(\frac{k-1}{2})}(s-u)^{\frac{k-3}{2}} \d u 
\\
&\hspace{-2mm}\le&\hspace{-3mm}
  \frac{R^{k+1}\Gamma(\frac 1 2)^{k-2}|x|^{\frac 1 2}}{\Gamma(\frac{k-1}{2})}(s-t)^{\frac k 2 -\frac 5 4}.
\ee* 
Finally, one obtains
\beq \label{ineqI33} 
&\hspace{-2mm}& \hspace{-3mm}
\left|\int_0^\infty\!\! \beta(s,y) \int_t^s\Delta^{(k)}_3(u)\d u\;\d y\right| \nonumber 
\\
&\hspace{-2mm}\le & \hspace{-3mm}
 \int_t^s\int_{-\infty}^\infty \big | \partial_t\big(g_0^f-g_0^{f'}\big)(t,0,u,z)\big|\d z \d u \int_0^\infty \big |  G_{f'}^{(k)}(u,z,s,y)-G_{f'}^{(k)}(u,0,s,y)\big|\d y \nonumber \nonumber 
 \\
&\hspace{-2mm}\le& \hspace{-3mm}
\frac{R^{k+1}\Gamma(\frac 1 2)^{k-2}}{\Gamma(\frac{k-1}{2})}(s-t)^{\frac{k}{2}-\frac{5}{4}} \|f-f'\|_s \int_t^s  (s-u)^{\frac{k}{2}-\frac{5}{4}} \d u\int_{-\infty}^\infty \big |z|^{\frac 1 2}|\partial_t g_0^{f}(t,0,u,z)-\partial_t g_0^{f'}(t,0,u,z)\big|\d z  \nonumber
\\
&\hspace{-2mm}\le& \hspace{-3mm}
\frac{R^{k+1}\Gamma(\frac 1 2)^{k-2}}{\Gamma(\frac{k-1}{2})}(s-t)^{\frac{k}{2}-\frac{5}{4}} \|f-f'\|_s\int_t^s   \d u\int_{-\infty}^\infty  \frac{R|z|^{\frac 1 2}}{(u-t)}\phi(r(u-t),z)\d z \nonumber 
\\
&\hspace{-2mm}\le& \hspace{-3mm}
\frac{R^{k+2}\Gamma(\frac 1 2)^{k-2}}{\Gamma(\frac{k-1}{2})} (s-t)^{\frac k 2-1}\|f-f'\|_s,
\eeq 
The required estimate \eqref{ineq:stab4} now follows by plugging \eqref{estim:last1}, \eqref{estim:last2}, \eqref{ineqI33} into \eqref{estim:last}.
\end{itemize}
\end{proof}

\begin{appendix}
\section{Technical results}
\label{sec_appendix}

\subsection{Definition and first properties of $c$}

\begin{lemma} \label{lemma:existence_c}
For every $(t,m)\in\R_+\times\Pc_1(\R_+)$, the equations \eqref{eq:charac_c} and \eqref{BnSigman} have unique solutions $c(t,m)$ and $c^n(t,m)$ satisfying 
\be* 
0\le c(t,m) \le \int_{\R_+} b(t, x,m)^+m(\d x) 
&\mbox{and}& 
0\le c^n(t,m) \le \int_{\R_+} b^n(t,x,m)^+m(\d x). 
\ee*
\begin{itemize}
\item[{\rm (i)}] Assume $|b(t,x,m)-b(t',x,m)|\le C|t-t'|^\delta$, uniformly in $(x,m)$, for some $C>0$, and $\delta\in (0,1)$. Then $|c(t,m)-c(t',m)|\le C|t-t'|^\delta$, uniformly in $(x,m)$. A similar statement holds for $b^n$ and $c^n$.
\item[{\rm (ii)}] If $b^n$ is Lipschitz in $x$, then there exists  $C\equiv C(n)>0$ such that
\be* 
|c^n(t,m)-c^n(t,m')|\le C\Big(1+\int_{\R_+} b^n(t,x,m)^+m(\d x)\Big)\Wc_1(m,m'),\quad \forall t\ge 0,~ m, m'\in\Pc_1(\R_+).
\ee* 
\item[{\rm (iii)}] $\lim_{n\to\infty}c^n(t,m)=c(t,m)$ holds for all $t\ge 0$ and $m\in \Pc_1(\R_+)$.
\end{itemize}
\end{lemma}
\begin{proof}
Consider the function $F(y):= \big(1 +m(H) \big)y - \int_{\R_+} \big(b(t,x, m) +y\big)^+H(x)m(\d x)$, $y\ge 0.$
By a straightforward verification, one finds $F(0)\le 0$, $F(\infty)=\infty$,   $F$ is strictly increasing and continuous on $\R_+$ by the dominated convergence theorem. Therefore, there exists a unique root, denoted by $c(t,m)\ge 0$, i.e. $F(c(t,m))=0$. In view of the inequality
\be*
c(t,m)
\;\le\; 
\frac{m(H) c(t,m)+\int_{\R_+}b(t, x,m)^+H(y)m(\d x)}{1+m(H)} 
\;\le\; \frac{m(H) c(t,m)+\int_{\R_+}b(t, x,m)^+m(\d x)}{1+m(H)},
\ee*
the bound for $c$ follows. The same argument applies for $c^n$. 

\vspace{1mm}

(i) For any $t,t'\ge 0$ and $m\in\Pc_1(\R_+)$, one has
\be* 
|c(t,m)-c(t,m')| 
&=& 
\frac{\big|\int_{\R_+} \!\!\big(b(t,x, m) +c(t,m)\big)^+H(x)m(\d x)-\int_{\R_+} \!\!\big(b(t',x, m) +c(t',m)\big)^+H(x)m(\d x)\big|}{1+m(H)} 
\\
&\le & \frac{m(H)}{1+m(H)}\big(C|t-t'|^\delta +|c(t,m)-c(t,m')|\big),
\ee*
which yields the required inequality. The  reasoning for $c^n$ is the same. 

\vspace{1mm}

(ii) Similar to (i), we compute the difference $c^n(t,m)-c^n(t,m')$ and obtain the desired inequality using the definitions of $c^n$ and $H^n$.

\vspace{1mm}

(iii) Again, we compute the difference $c(t,m)-c^n(t,m)$ and obtain
\be* 
&&\big(1+m(H)\big)|c(t,m)-c^n(t,m)| \\
&=& \left | 
\int_{\R_+} \big(b(t,x, m) +c(t,m)\big)^+H(x)m(\d x)-\frac{1+m(H)}{1+m(H^n)} 
\int_{\R_+} \big(b^n(t,x, m) +c^n(t,m)\big)^+H^n(x)m(\d x)\right| \\
&\le & |m(H)-m(H^n)|\int_{\R_+} \big(b(t,x, m) +c(t,m)\big)^+H(x)m(\d x) \\
&& +\frac{1+m(H)}{1+m(H^n)} \bigg[\int_{\R_+} |b(t,y,m)-b^n(t,y,m)|H(y)\d y + m(H)|c(t,m)-c^n(t,m)| \\
&& +\int_{\R_+} (b^n(t,y,m)-c^n(t,m))^+|H(y)-H^n(y)|\d y \bigg].
\ee*
We may conclude by the dominated convergence theorem by using the fact that $b^n$ converges uniformly to $b$ and $b$ is of linear growth on $x$.
\end{proof}

\subsection{On hitting times of It\^o processes}

\begin{lemma}\label{lem:ito-martingale}
Let $Z$ be an It\^o-process given by 
\be* 
Z_t=Z_0+\int_0^t k(u,Z_u)\d u+\int_0^t a(u,Z_u)\d W_u,\quad \forall t\ge 0
\ee*
where $Z_0>0$, $k$ is bounded and $a$ takes values in some interval $[\underline a,\overline a]\subset (0,\infty)$. Define its first hitting time at zero $\tau:=\inf\{t\ge 0: Z_t\le 0\}$ and running minimum $\underline{Z}_t:=\min_{0\le s\le t} Z_s$. Then,
\begin{itemize}
 
    \item[\rm(i)] $\P[\underline{Z}_t=0]=0$ for all $t\ge 0$;
    \item[\rm(ii)] The map $\R_+\ni t  \mapsto\P[\tau>t]$ is strictly decreasing and takes values in $(0,1]$;

    \item[\rm(iii)] If in addition the maps $k$ and $a$ are Lipschitz in $x$, uniformly in $t$, then
\be* 
\big|\mathbb P[\tau>t]-\mathbb P[\tau>s]\big|
\le 
C\E[Z^{-2}_0]|t-s|^{1/6},
~~\mbox{for all}~~t,s\ge 0,\mbox{ for some }C.
\ee*  
\end{itemize}
\end{lemma}
\begin{proof}
By introducing the equivalent probability measure $\Q$ defined via
\be*
 \frac{\mathrm{d}\Q}{\mathrm{d}\P}\Big |_{\Fc_t}
 =
 \exp\Big({-\int_0^t \lambda_u\mathrm{d}W_u-\frac12\int_0^t  \lambda_u^2\mathrm{d}u}\Big)=:\Ec_t(-\lambda),\quad \mbox{with } \lambda_t:=\frac{k(t,Z_t)}{a(t,Z_t)},
\ee* 
the dynamics can be rewritten as $\d Z_t=a(t,Z_t)\d W^\Q_t$ in view of Girsanov's theorem, where $W^\Q$ is a Brownian motion under $\Q$. 

\vspace{1mm}

To prove {\rm (iii)}, we compute for $s=t+h\ge t\ge 0$ that 
\be*
\Q[\tau>t]-\Q[\tau>s] 
&=&  
\mathbb Q\Big[\underline Z_t>0, \inf_{t\le v\le s}\int_t^v a(r,Z_r)\d W^\Q_r\le -Z_t\Big] \\
&\le &\Q\Big[Z_t>0, \inf_{t\le v\le s}\int_t^v a(r,Z_r)\d W^\Q_r\le -Z_t\Big] \\
&=&\int_{(0,\infty)}
\Q\Big[\inf_{t\le v\le s}\int_t^v a(r,Z_r)\d W^\Q_r \le-x\Big |Z_t=x\Big] \Q[Z_t\in \d x] \\
&\le& \int_{(0,\infty)} \min\Big(1, x^{-2}\E\Big[\Big|\int_t^{s}a(r,Z_r)\d W^\Q_r\Big|^2 \Big|Z_t=x\Big] \Big)\Q[Z_t\in \d x] \\
&\le& \int_{(0,\infty)}\min\Big\{1, \frac{h{\overline a}^2}{x^2}\Big\}\Q[Z_t\in \d x],
\ee*
where the second last inequality follows from Doob's inequality. By \cite[Theorem 2.5]{KUSUOKA2017359}, $Z_t$ has a density denoted by $q_t$ satisfying $
q_t(x) \le \frac{C'}{\sqrt{t}}e^{-Kx^2/t}$ for all $x\in\R$, where $C'>K>0$ are constants depending only on $\underline a, \overline a, |a_x|_{\L^\infty}$. Hence,
\be*
\Q[\tau>t]-\Q[\tau>s] 
\le  
\int_0^{\infty} \min\Big\{1, \frac{h{\overline a}^2}{x^2}\Big\} \frac{C'}{\sqrt{t}}e^{-Kx^2/t} \d x
\le \frac{C'\sqrt{h}}{\sqrt{t}} \int_0^{\infty} \min\Big\{1,\frac{\overline a^2}{y^2}\Big\} \d y \le \frac{C'\sqrt h}{\sqrt t}.
\ee* 
On the other hand, we note that $
\Q[\tau>t|Z_0=z]-\Q[\tau>s|Z_0=z] \le 1-\Q[\tau>s|Z_0=z]=\Q[\tau\le s|Z_0=z]\le {\overline a}^2s/z^2$ 
by Doob's inequality. Therefore,
\be* 
\Q[\tau>t]-\Q[\tau>s] \le \E\Big[\min\Big(\frac{C'\sqrt h}{\sqrt t}, \frac{{\overline a}^2s}{Z_0^2}\Big)\Big]\le
\E\Big[ \frac{C''}{Z_0^2}\Big]h^{1/3},
\ee*
where the existence of $C''$ follows from a straightforward computation by distinguishing the three cases $h\ge1$, $t\le h^{1/3}<1$ and $t>h^{1/3}$. Finally, one has
\be*
\P[\tau>t]-\P[\tau>s]
&=&
\E^{\Q}\Big[\frac{1}{\Ec_s(-\lambda)}
                  \big({\1}_{\{\tau> t\}} 
                          -  {\1}_{\{\tau>s\}}\big)
           \Big] \\
         &  \le &  \exp\Big(\frac{3}{2}\int_0^s\lambda_u^2\d u\Big) \big(\Q\big[\tau> t\big] - \Q\big[\tau> s\big]\big)^{1/2}
\;\le\; \E\Big[ \frac{C}{Z_0^2}\Big]h^{1/6},
\ee*
where the second inequality follows from the H\"older inequality together with the boundedness of $\lambda$.

\vspace{1mm}

By the equivalence between $\P$ and $\Q$, we now prove {\rm (i)} by showing that $\Q[\underline{Z}_t=0]=0$, where $\underline{Z}_{s,t}:=\min_{s\le r\le t} Z_r$. Notice that
\be*
\Q\big[\underline{Z}_t=0 \big] 
=  
\Q\big[\tau\le t, \underline{Z}_{\tau,t}\ge 0\big] 
=  
\Q\big[\tau<t, \underline{Z}_{\tau,t}\ge 0\big] 
=  \Q\Big[\tau<t, \min_{0\le u\le t-\tau}(Z_{\tau+u}-Z_{\tau})\ge 0\Big], 
\ee*
where the second equality holds as $\Q[\tau=t]=\P[\tau=t]=0$ by {\rm (i)}. Conditioning on $\{\tau<t\}$, $(Z_{u+\tau}-Z_{\tau})_{u\ge 0}$ is a time-changed Brownian motion, i.e. $Z_{u+\tau}-Z_{\tau}=\beta_{\int_{0}^{u}\!\!\sigma_{\tau+s}^2\d s}$, and we deduce from the fact that $\sigma$ is bounded from below away from zero by some constant $\underline{a}$ that 
\be*
\min_{0\le u\le t- \tau} (Z_{u+\tau}-Z_{\tau}) = \min_{0\le u\le t-\tau}\beta_{\int_{0}^{u}\sigma_{\tau+s}^2\d s} \ge \min_{0\le r\le {\underline a}^2(t-\tau)}\beta_r.
\ee* 
Then, it follows from the trajectorial properties of the Brownian motion that
\be*
\Q\big[\underline{Z}_t=0 \big]  
\le 
\Q\Big[\tau<t,  \min_{0\le r\le {\underline a}^2(t-\tau)}\beta_r\ge 0 \Big] =0.
\ee*
It remains to prove {\rm (ii)}. Suppose to the contrary that the map $t\mapsto\P[\tau>t]$ is flat on some open interval $(t,s)$, then $
\Q\big[\tau>t, \underline{Z}_{t,s} \le 0\big]=\P\big[\tau>t, \underline{Z}_{t,s} \le 0\big]=0.$ In particular, one has $\Q\left[\underline{Z}_{t,s} \le  0|\tau>t\right]=0$. Using the above time-change argument, one has
\be* 
0= \Q\Big[\underline{Z}_{t,s} \le  0|\tau>t\Big] 
= 
\Q\Big[\min_{0\le u\le s-t}\beta_{\int_{0}^{u}\sigma_{t+r}^2\d r} 
            \le -Z_t\Big|\tau>t\Big]
\ge  \Q\Big[\min_{0\le r\le {\underline a}^2(s-t)}\beta_r \le  -Z_t\Big|\tau>t\Big] >0, 
\ee* 
contradiction~! To show $\P[\tau>t]>0$ it suffices to show $\Q[\tau>t]>0$ and we use the same time-change argument. Namely, $\Q[\tau>t] 
= 
\Q\Big[\min_{0\le u\le t}\beta_{\int_{0}^{u}\sigma_{r}^2\d r} > -Z_0\Big]
\ge
\Q\Big[\min_{0\le r\le {\overline a}^2t}\beta_r \le  -Z_0\Big] >0$. 
\end{proof} 

\subsection{Convergence of sequence of SDEs starting at different initial times}

This section justifies the convergence of a sequence of processes starting at different times. For simplicity, we restrict to driftless SDEs with scalar continuous paths in $\Cc:=C([0,T])$:
\begin{eqnarray*}
    \Lc(R^{n,\beta^n(t)}_0)=\nu,
    &\mbox{and}&
    R^{n,\beta^n(t)}_s
    =
    R^{n,\beta^n(t)}_0
    +
    \int_{0}^s a^n\big(r, R^{n,\beta^n(t)}_r \big)^{\frac{1}{2}}\1_{\{r \ge \beta^n(t)\}}\; \mathrm{d}W_r,
    ~~s \in [0,T].
\end{eqnarray*}
where, for all $n \ge 1$, $\beta^n:[0,T]\longrightarrow[0,T]$ and $a^n[0,T] \x \R \longrightarrow[\eta,\infty)$, for some $\eta>0$, are Borel bounded maps, Lipschitz in $x$ uniformly in $t$, and $\nu \in \Pc_q(\R)$, for some $q >1$.

\begin{proposition} \label{prop:limit_var_initial_time}
Assume $\beta^n\!\underset{n\to\infty}{\longrightarrow}\!\beta$ and $a^n\!\underset{n\to\infty}{\longrightarrow}\! a$ a.e. and let $\mu^n(\mathrm{d}\xb,\mathrm{d}t)\!:=\!\Lc(R^{n,\beta^n(t)})(\mathrm{d}\xb)\mathrm{d}t$. Then,
\\
{\rm (i)} the sequence of measures $(\mu^n)_{n \ge 1}$ is relatively compact in $\Wc_q$;
\\
{\rm (ii)} The limit $\mu$ of any convergent subsequence satisfies $\mu(\mathrm{d}\xb,\mathrm{d}t)=\Lc(R^{\beta(t)})(\mathrm{d}\xb)\;\mathrm{d}t$ where the process $R^{\beta(t)}$ satisfies, for a.e. $t \in [0,T]$,
\begin{eqnarray*}
    \Lc(R^{\beta(t)}_0)=\nu,
    &\mbox{and}&
    R^{\beta(t)}_s
    =
    R^{\beta(t)}_0
    +
    \int_{0}^s a\big(r, R^{\beta(t)}_r \big)^{\frac{1}{2}}\1_{\{r \ge \beta(t)\}}\; \mathrm{d}W_r,
    ~~\mbox{for all}~s \in [0,T].
\end{eqnarray*}
\end{proposition}

\begin{proof} We proceed in 3 steps.

\medskip
\noindent \emph{Step {\rm 1}: Relative compactness.} Notice that $\sup_{n \ge 1} \sup_{t \in [0,T]}\E\Big[\sup_{r \in [0,T]} |R^{n,\beta^n(t)}_r|^q \Big]< \infty$ and
\begin{align*}
    \lim_{|s-s'| \to 0}\sup_{n \ge 1}\sup_{t \in [0,T]}\E \Big[ \big| R^{n,\beta^n(t)}_{s} - R^{n,\beta^n(t)}_{s'} \big|^q \Big]=0.
\end{align*}
Then, $\big(\mu^n(\mathrm{d}\xb,\mathrm{d}t)\big)_{n \ge 1}$ is relatively compact in $\Wc_q$, and converges to some limit $\mu(\mathrm{d}\xb,\mathrm{d}t)$, after possibly passing to a subsequence. 

\medskip
\noindent \emph{Step {\rm 2}} We next prove the following technical result which will be used in the next step to identify the nature of this limit $\mu$. Let $(\psi^n:[0,T] \x [0,T] \x \R \to \R)_{1 \le n \le \infty}$ be a sequence of bounded Borel maps continuous in the third variable s.t. $\Lim_{n \to \infty} \psi^n=\psi^\infty$, a.e. Then, we claim that
    \begin{eqnarray*}
        \Psi_n
        :=
        \E \bigg[\int_0^T\!\!\! \int_{\beta^n(t)}^T \psi^n\big(t,r, R_r^{n,\beta^n(t)} \big)\;\mathrm{d}r\;\mathrm{d}t \bigg] 
        &\underset{n\to\infty}{\longrightarrow}&
        \int_{[0,T]^2} \int_{\Cc} \psi^\infty(t,r,\xb_r) \1_{\{r \ge \beta(t)\}} \mu(\mathrm{d}\xb,\mathrm{d}t)\mathrm{d}r.
    \end{eqnarray*}
To see this, notice first that, by the a.e. convergence of $\beta^n$ to $\beta$ together with the continuity of $\psi^n$ in the third variable, it follows from the weak convergence of $(\mu^n)_{n \ge 1}$ that for all $k \ge 1$:
    \begin{align*}
        \Psi_n
        =
        \int_{[0,T]^2} \int_{\Cc} \psi^k(t,r,\xb_r) 
            \1_{\{r \ge \beta^n(t)\}} \mu^n(\mathrm{d}\xb,\mathrm{d}t)\mathrm{d}r
        \underset{n\to\infty}{\longrightarrow}
        \int_{[0,T]^2} \int_{\Cc} \psi^k(t,r,\xb_r) \1_{\{r \ge \beta(t)\}} \mu(\mathrm{d}\xb,\mathrm{d}t)\mathrm{d}r.
    \end{align*}
Then, in order to prove the required result, we now show that 
\begin{eqnarray}\label{eq:psinpsik}
        \lim_{k \to \infty}\lim_{n \to \infty}
        \varpi_{k,n}=0,
        &\mbox{where}&
        \varpi_{k,n}:=
        \int_{[0,T]^2} \int_{ \Cc} |\psi^n-\psi^k|(t,r,\xb_r) \1_{\{r \ge \beta^n(t)\}} \mu^n(\mathrm{d}\xb,\mathrm{d}t)\mathrm{d}r.
    \end{eqnarray}
 For an arbitrary $K>0$, we first estimate by the Chebychev inequality that 
    $$
    \varpi_{k,n}
    \le 
    \int_{[0,T]^2} \int_{\Cc} |\psi^n-\psi^k|(t,r,\xb_r) \1_{\{r \ge \beta^n(t)\}} \1_{\{|\xb_r| \le K\}} \mu^n(\mathrm{d}\xb,\mathrm{d}t)\mathrm{d}r+ C \;\; \frac{\E\big[\sup_{r \in [0,T]} |R^{n,\beta^n(t)}_r| \big]}{K}.
    $$
Since  $a^n$ is uniformly bounded from below above zero, it follows from \cite[Theorem 6.3.1--(i)]{FK-PL-equations} that we may find a constant $C>0$, independent of $n$, such that 
$$
\varpi_{k,n}
\le 
C\! \int_{[0,T]^2 \x [0,1] \x \R} \1_{\{|x| \le K\}}|\psi^n-\psi^k|(t,r,x) \;\mathrm{d}x\;\mathrm{d}r\;\mathrm{d}t + C \;\; \frac{\sup_{n \ge 1} \sup_{t \in [0,T]}\E\big[\sup_{r \in [0,T]} |R^{n,\beta^n(t)}_r| \big]}{K}.
$$
The required result \eqref{eq:psinpsik} follows by taking the limits $n \to \infty$, then $k \to \infty$, and finally $K \to \infty$.

\medskip
\noindent \emph{Step {\rm 3}: Identification of the limit.} 
    It is obvious that $\mu(\Cc \x [0,t])=t$ for all $t \in [0,T]$. Then, we can write $\mu(\mathrm{d}\xb,\mathrm{d}t)=\mu^t(\mathrm{d}\xb)\mathrm{d}t$ where the map $ t\in [0,T] \longmapsto \mu^t \in \Pc(\Cc)$ is Borel measurable. For any bounded map $g:[0,T] \to \R$ and twice differentiable map $f:\R \to \R$, it follows from It\^o's formula that
    $$
        \int\!\! g(t)f(\xb_r)\mu^n(\mathrm{d}\xb,\mathrm{d}t)
        =\!\!
        \int\!\! g(t)f(x)\nu(\mathrm{d}x)\;\mathrm{d}t
        +
        \frac{1}{2}\int_0^r\!\!\!  \int\!\! g(t) f''(\xb_u) a^n(u,\xb_u) \1_{\{u \ge \beta^n(t)\}} \mu^n(\mathrm{d}\xb,\mathrm{d}t)\;\mathrm{d}u,
    $$
for all $r \in [0,T]$ and $n \ge 1$. Notice that $\Lim_{n \to \infty} g(t) f''(x) a^n(u,x)=g(t) f''(x) a(u,x)$ for a.e. $(s,t,x,u)$. We may then take the limit in the last equality by using the technical result of Step 2, and obtain:
    $$
        \int g(t)f(\xb_r)\mu(\mathrm{d}\xb,\mathrm{d}t)
        =
        \int g(t)f(x)\nu(\mathrm{d}x)\;\mathrm{d}t
        + \frac{1}{2}
        \int_0^r \!\!\! \int g(t) f''(\xb_u) a(u,\xb_u) \1_{\{u \ge \beta(t)\}} \mu(\mathrm{d}\xb,\mathrm{d}t)\;\mathrm{d}u.
    $$
By the arbitrariness of $g$, this implies that
    $$
        \int_{\Cc} f(\xb_r)\mu^t(\mathrm{d}\xb)
        =
        \int_{\R} f(x)\nu(\mathrm{d}x)
        + \frac{1}{2}
        \int_0^r \!\!\! \int_{\Cc} f''(\xb_u) a(u,\xb_u) \1_{u \ge \beta(t)} \mu^t(\mathrm{d}\xb)\;\mathrm{d}u,
        ~\mbox{for a.e.}~t \in [0,T],
    $$
completing the proof by equivalence between the Fokker--Planck equation and the corresponding SDE.
\end{proof}

\end{appendix}

\bibliographystyle{plain}

\bibliography{bibliographyFabrice_24-08-22}

\end{document}